\documentclass[12pt]{tac}
\title{On cyclic star-autonomous categories} 
\author{J.M.~Egger and M.B.~McCurdy}
\thanks{%
  First author supported by EPSRC research grant EP/F042043/1. 
  Second author supported by MQRES scholarship.}
\address{%
  School of Informatics, University of Edinburgh
  \\ Informatics Forum, 10 Crichton Street, Edinburgh, UK EH8 9AB
  \\[5pt] Department of Computing, Macquarie University
  \\ North Ryde, NSW, Australia 2109}
\copyrightyear{2010}
\eaddress{jeffegger@yahoo.ca\CR micah.mccurdy@gmail.com}

\usepackage{amssymb}      
\usepackage{bbold}        
\usepackage{doi}          
\usepackage{proof}        
\usepackage{stmaryrd}     
\usepackage{pstricks,pst-plot} 
\psset{unit=4pt}
\psset{labelsep=0pt} 
\psset{framearc=.5} 
\newgray{obvcol}{.9}
\newgray{revcol}{.5}
\newgray{boxcol}{1} 
\usepackage[all]{xy}\UseComputerModernTips

\newtheorem{defn}{Definition}
\newtheorem{defs}{Definitions}
\newtheorem{lem}{Lemma}
\newtheorem{thm}{Theorem}
\newtheorem{cor}{Corollary}
\newtheoremrm{rem}{Remark}
\newtheoremrm{eg}{Example}

\let\tens\varowedge   
\let\parr\varovee     

\def\relcirc{\mathrel{\circ}}   
\def\loll{\mathrel{\relbar\joinrel\relcirc}}  
\def\llol{\mathrel{\relcirc\joinrel\relbar}}  

\let\psymbol\ast
\newlength{\prepfu} 
\def\prep#1{\mathchoice
  {\settoheight{\prepfu}{\ensuremath{\displaystyle#1}}
    \rule{0pt}{\prepfu}^\psymbol\mkern-1mu{#1}}
  {\settoheight{\prepfu}{\ensuremath{\textstyle#1}}
    \rule{0pt}{\prepfu}^\psymbol\mkern-1mu{#1}}
  {\settoheight{\prepfu}{\ensuremath{\scriptstyle#1}}
    \rule{0pt}{\prepfu}^\psymbol\mkern-1mu{#1}}
  {\settoheight{\prepfu}{\ensuremath{\scriptscriptstyle#1}}
    \rule{0pt}{\prepfu}^\psymbol\mkern-1mu{#1}}}
\def\perp#1{\mathchoice
  {\settoheight{\prepfu}{\ensuremath{\displaystyle#1}}
    {#1}\mkern-0mu\rule{0pt}{\prepfu}^\psymbol}
  {\settoheight{\prepfu}{\ensuremath{\textstyle#1}}
    {#1}\mkern-0mu\rule{0pt}{\prepfu}^\psymbol}
  {\settoheight{\prepfu}{\ensuremath{\scriptstyle#1}}
    {#1}\mkern-0mu\rule{0pt}{\prepfu}^\psymbol}
  {\settoheight{\prepfu}{\ensuremath{\scriptscriptstyle#1}}
    {#1}\mkern-0mu\rule{0pt}{\prepfu}^\psymbol}}

\let\qpsymbol\bot
\def\qprep#1{\mathchoice
  {\settoheight{\prepfu}{\ensuremath{\displaystyle#1}}
    \rule{0pt}{\prepfu}^\qpsymbol\mkern-1mu{#1}}
  {\settoheight{\prepfu}{\ensuremath{\textstyle#1}}
    \rule{0pt}{\prepfu}^\qpsymbol\mkern-1mu{#1}}
  {\settoheight{\prepfu}{\ensuremath{\scriptstyle#1}}
    \rule{0pt}{\prepfu}^\qpsymbol\mkern-1mu{#1}}
  {\settoheight{\prepfu}{\ensuremath{\scriptscriptstyle#1}}
    \rule{0pt}{\prepfu}^\qpsymbol\mkern-1mu{#1}}}
\def\qperp#1{\mathchoice
  {\settoheight{\prepfu}{\ensuremath{\displaystyle#1}}
    {#1}\mkern-0mu\rule{0pt}{\prepfu}^\qpsymbol}
  {\settoheight{\prepfu}{\ensuremath{\textstyle#1}}
    {#1}\mkern-0mu\rule{0pt}{\prepfu}^\qpsymbol}
  {\settoheight{\prepfu}{\ensuremath{\scriptstyle#1}}
    {#1}\mkern-0mu\rule{0pt}{\prepfu}^\qpsymbol}
  {\settoheight{\prepfu}{\ensuremath{\scriptscriptstyle#1}}
    {#1}\mkern-0mu\rule{0pt}{\prepfu}^\qpsymbol}}

\def\inverse{^{-1}}
\def\teenyleft{_{\scriptscriptstyle\shortleftarrow}}
\def\teenyright{_{\scriptscriptstyle\shortrightarrow}}

\def\idsymbol{\mathrm{id}}
\let\assocsymbol\alpha
\let\braidsymbol\beta
\let\lulawsymbol\lambda
\let\rulawsymbol\rho
\def\tbraidsymbol{\hat\braidsymbol} 
\def\pbraidsymbol{\check\braidsymbol} 
\let\cadsymbol\kappa
\def\llcadsymbol{{\stackrel{\teenyright}{\cadsymbol}}}
\def\rrcadsymbol{{\stackrel{\teenyleft}{\cadsymbol}}}
\let\balancesymbol\zeta
\let\cyclesymbol\varphi
\let\Cyclesymbol\Phi 
\let\stitchsymbol\xi
\def\nubalancesymbol{\balancesymbol^{\cyclesymbol}}
\def\nunubalancesymbol{\balancesymbol^{\nucyclesymbol}}
\def\nucyclesymbol{\cyclesymbol^{\balancesymbol}}
\def\nuCyclesymbol{\Cyclesymbol^{\balancesymbol}}
\def\nunuCyclesymbol{\Cyclesymbol^{\nubalancesymbol}}

\newcommand{\id}[1][]{\idsymbol_{#1}}
\newcommand{\assoc}[1][]{\assocsymbol_{#1}}
\newcommand{\braid}[1][]{\braidsymbol_{#1}}
\newcommand{\lulaw}[1][]{\lulawsymbol_{#1}}
\newcommand{\rulaw}[1][]{\rulawsymbol_{#1}}
\newcommand{\tbraid}[1][]{\tbraidsymbol_{#1}}
\newcommand{\pbraid}[1][]{\pbraidsymbol_{#1}}
\newcommand{\llcad}[1][]{\llcadsymbol_{#1}}
\newcommand{\rrcad}[1][]{\rrcadsymbol_{#1}}
\newcommand{\balance}[1][]{\balancesymbol_{#1}}
\newcommand{\cycle}[1][]{\cyclesymbol_{#1}}
\newcommand{\Cycle}[1][]{\Cyclesymbol_{#1}}
\newcommand\stitch[1][]{\stitchsymbol_{#1}}
\newcommand\nubalance[1][]{\nubalancesymbol_{#1}}
\newcommand\nunubalance[1][]{\nunubalancesymbol_{#1}}
\newcommand\nucycle[1][]{\nucyclesymbol_{#1}}
\newcommand\nuCycle[1][]{\nuCyclesymbol_{#1}}
\newcommand\nunuCycle[1][]{\nunuCyclesymbol_{#1}}

\newcommand\canon[1][]{\eta} 
\newcommand\demorgan[1][]{\vartheta} 
\newcommand\ioCycle[1][]{\Cycle^{\id}}
\let\zangify\tilde
\newcommand\zangcycle[1][]{\zangify\cyclesymbol}

\let\cedille\c

\def\a{\mathbb{a}}
\def\b{\mathbb{b}}
\def\c{\mathbb{c}}
\let\d=d
\let\e=e
\def\g{\mathbb{g}}
\def\gexplained{((g,\leq),\cdot,\gunit,\blank\inverse)}
\let\j=i
\let\k=k
\def\p{\mathbb{p}}
\def\pexplained{(p,\leq)}

\def\C{\mathcal{C}}
\def\F{\Z_2} 
\def\J{\mathcal{J}}
\def\K{\mathcal{K}}
\def\Kexplained{(\K,{\tens},\e,{\parr},\d,\perp\blank,\prep\blank)}
\def\V{\mathcal{V}}
\def\Vexplained{(\V,{\tens},\e,{\parr},\d,\perp\blank,\prep\blank,\cycle)}
\def\tV{(\V,{\tens},\e)}
\def\Z{\mathbb{Z}}

\let\gunit\eta      
\let\gdo\varepsilon 
\let\copi\upsilon
\def\unmu{\mu_\circ} 
\def\unnu{\nu_\circ} 

\let\newt\varotimes   
\let\jnewt\newt 
\let\knewt\newt
\def\vnewt{\newt_k}

\let\of\circ			
\def\fo{\mathrel{;}}		

\newcommand{\Id}[1][]{\mathrm{Id}_{#1}}
\newcommand{\name}[2][]{{}^\ulcorner{#2}{}^{\urcorner#1}}
\def\lCurry#1{\mathsf{lCurry}(#1)}
\def\rCurry#1{\mathsf{rCurry}(#1)}
\def\invlCurry#1{\mathsf{lCurry}\inverse(#1)}
\def\invrCurry#1{\mathsf{rCurry}\inverse(#1)}
\def\lbind{\mathsf{lbind}}
\def\rbind{\mathsf{rbind}}

\newcommand{\lact}[1][-,-,-]{\left\langle{#1}\right\rangle_{\fo}}
\newcommand{\ract}[1][-,-,-]{\left\langle{#1}\right\rangle_{\fo}}

\def\Zang#1{\mathbf{Adj}_{\Z}(#1)}
\def\Fang#1{\mathbf{Adj}_{\F}(#1)}
\def\Vecfd{\mathbf{Vec}_{\rm fd}}
\newcommand{\Profsymbol}[1][]{\mathfrak{Prof}_{#1}}
\newcommand{\Prof}[2][]{\mathfrak{Prof}_{#1}(#2)}

\let\to\longrightarrow
\let\To\Rightarrow
\def\arrow#1:#2->#3;{#2 \stackrel{#1}{\longrightarrow} #3}
\def\proarrow#1:#2->#3;{#2 \stackrel{#1}{\nrightarrow} #3}
\def\isoto{\stackrel{_\sim}{\longrightarrow}}
\let\iso\cong			
\def\blank{{(~)}}
\def\op{^{\mathsf{op}}}
\def\rv{^{\mathsf{rev}}}
\def\ob{\mathsf{ob}\,}
\newcommand{\xhom}[3][]{\langle#2,#3\rangle_{#1}}
\let\meet\wedge			
\let\implies\Rightarrow		
\let\entails\Rightarrow 
\def\iff{\Leftrightarrow}
\def\evan#1{(#1)} 
\let\card\#
\def\one{\mathbb{1}}
\def\two{2} 
\def\classof#1{\left\{ #1 \right\}}
\newcommand{\familyof}[2][]{\classof{#2}_{#1}}

\newenvironment{diagram}{
  \displaymath
  
  \xymatrix@1@M=3pt@C=5pc}{
  \enddisplaymath}
\newenvironment{inline}{
  \math
  
  \xymatrix@1@M=3pt@C=3pc}{
  \endmath}
\def\wonky#1#2#3{\ar`#1[#2] `[#2]#3 [#2]}  
\def\fixwonky#1#2#3{\ar@{-<}`#1[#2] `[#2]#3 [#2]}
\def\invwonky#1#2#3{\ar@{<-}`#1[#2] `[#2]#3 [#2]}
\def\fixinvwonky#1#2#3{\ar@{>-}`#1[#2] `[#2]#3 [#2]}

\def\nt{natural transformation}
\def\ni{natural isomorphism}
\def\cni{coherent \ni}

\def\ld{linearly distributive}
\def\staut{star-autonomous} 
\def\arbstaut{arbitrary \staut}
\def\brdstaut{braided \staut}
\def\cycstaut{cyclic \staut}
\def\symstaut{symmetric \staut}

\def\boring{compact}
\def\boringness{compactness} 

\def\cthing{cycle} 
\def\cprop{cyclicity} 
\def\ts{$\tens$-semi}
\def\ps{$\parr$-semi}
\def\kq{quasi}

\def\zang{$\Z$-string}
\def\fang{$\F$-string}

\def\etc{\emph{etc.}}
\def\Thatis{That is}	
\def\wrt{with respect to}	

\def\scary#1{``#1''} 		

\def\kimmo{\ensuremath{(\textsc{r})}} 
\def\kprime{\ensuremath{(\textsc{inv})}} 
\def\tnul{\ensuremath{(\textsc{t}_0)}}
\def\tbin{\ensuremath{(\textsc{t}_2)}} 
\def\pnul{\ensuremath{(\textsc{p}_0)}} 
\def\pbin{\ensuremath{(\textsc{p}_2)}} 
\def\blrzero{\ensuremath{(\textsc{blr}_0)}} 
\def\blrtwo{\ensuremath{(\textsc{blr}_2)}} 
\def\mzero{\ensuremath{(\textsc{m}_0)}} 
\def\mtwo{\ensuremath{(\textsc{m}_2)}} 
\def\varmtwo{\ensuremath{(\textsc{m}_2^{-1})}} 
\def\etwo{\ensuremath{(\textsc{e})}} 
\def\eagain{\ensuremath{(\textsc{e}^{-1})}} 
\def\naty{\ensuremath{(\textsc{n})}} 
\def\vartbin{\ensuremath{(!)}}
\def\qpnul{\ensuremath{(?)}}
\def\qkimmo{\ensuremath{(??)}}
\def\tsbax{\ensuremath{({\textsc{\^b}})}}
\def\psbax{\ensuremath{({\textsc{\v{b}}})}}

\def\Rep#1{\mathbf{mod}_{#1}} 
\begin{document}

\maketitle

\begin{abstract}
We discuss \emph{cyclic star-autonomous categories}; 
that is, unbraided \staut\ categories in which the left and right
duals of every object $p$ are linked by coherent natural isomorphism.  
We settle coherence questions which have arisen concerning such
cyclicity isomorphisms, and we show that such cyclic structures are
the natural setting in which to consider enriched profunctors. 
Specifically, if $\V$ is a cyclic star-autonomous category, then the
collection of $\V$-enriched profunctors carries a canonical cyclic
structure. 
In the case of braided star-autonomous categories, we discuss the
correspondences between cyclic structures and balances or tortile
structures. 
Finally, we show that every cyclic star-autonomous category is
equivalent to one in which the cyclicity isomorphisms are identities.  
\end{abstract}

\section{Introduction (Overview)} \label{sec:intr}
In an \arbstaut\ category 
(in particular, one which is not necessarily symmetric \cite{Bar'95}), 
every object $p$ has two duals: 
commonly, one is denoted $\perp{p}$ or $\qperp{p}$, 
the other $\prep{p}$ or $\qprep{p}$.  
(We prefer $\perp{}$ over $\qperp{}$ except, for obscure reasons, 
in the posetal examples below.)  
A \emph{\cycstaut\ poset} is defined in \cite{Yet'90} to be a
\staut\ poset with the property that these two duals always coincide.
It is important to note that, even in the posetal case, cyclicity is a
much weaker phenomenon than symmetry. 

\begin{eg} \label{eg:group}
  It is well-known that any ordered group $\g=\gexplained$
  determines a closed monoidal poset with 
  $\alpha\loll\beta:=\alpha\inverse\cdot\beta$ and 
  $\beta\llol\alpha:=\beta\cdot\alpha\inverse$;
  it follows that {every} element of the group is \emph{dualising}.
  Thus every \emph{pointed ordered group} 
  $(\g,\gdo)$, for $\gdo$ a fixed but arbitrary element of $g$, 
  determines a \staut\ poset. 
  (This example is called an \emph{ordered shift group} in
  \cite{CocSee:ldc}.) 
  Since $\qperp\omega=\omega\loll\gdo$ and 
  $\qprep\omega=\gdo\llol\omega$, we have that 
  $(\g,\gdo)$ is cyclic if and only if $\gdo$ is central.
  In particular, $(\g,\gunit)$ 
  is always cyclic---but it is symmetric if and only if $\g$ is abelian.  
\end{eg} 

\begin{eg} \label{eg:rel}
  Binary relations $\proarrow:s->s;$ 
  (where $s$ is some fixed but arbitrary set) 
  form a \cycstaut\ poset: 
  the tensor product is the usual composition of relations, 
  and the dualising object is the complement of the equality relation
  ($\neq$).  
  It is routine to verify that 
  \[ \qperp\omega = (\omega\loll{\neq}) = \neg\omega\rv
  = ({\neq}\llol\omega) = \qprep\omega \]
  holds for every $\omega : \proarrow:s->s;$.
  But, of course, symmetry can only occur when $\card{s}\leq1$. 
  (Here, $\omega\rv$ denotes the \emph{reverse} of $\omega$ 
  which is more commonly denoted $\omega\op$.)  

  More generally, $\two$-valued profunctors $\proarrow:\p->\p;$ 
  (where $\p=\pexplained$, some fixed but arbitrary poset) form a
  \cycstaut\ poset:  
  the tensor product is the usual composition of profunctors, 
  and the dualising object is the complement of the reverse ordering
  ($\ngeq$).  
  It is again routine to verify that 
  \[ \qperp\omega=(\omega\loll{\ngeq})=\neg\omega\rv
  =({\ngeq}\llol\omega)=\qprep\omega \]
  holds for every $\omega : \proarrow:\p->\p;$.
  Again, symmetry can only occur when $\card\p\leq1$. 
  (Observe that, in general, neither $\neg\omega$ nor $\omega\rv$ 
  is a profunctor $\proarrow:\p->\p;$, but that $\neg\omega\rv$ is.)  
\end{eg}

It is well-understood, at least in principle, that the term 
\emph{\cycstaut\ category} should mean a \staut\ category equipped
with a \cni\ $\arrow:\perp{p}->\prep{p};$. 
But this raises the question: what are the right coherence axioms?
This question is complicated by the fact that there is a second
approach to the phenomenon of cyclicity which does not explicitly
refer to dual objects.  
(This is not a new observation; on the contrary, the origin of the
term \emph{cyclic} is tied up in this approach---see again,
\cite{Yet'90}.) 

Since there are several equivalent definitions of \staut\ category, 
let use make clear that we use the one advocated in \cite{CocSee:ldc}: 
a linearly distributive category with chosen left and right duals for
every object.    
We generally use $\tens$ for \emph{tensor} and $\parr$ for \emph{par}; 
the linear distributions 
$\arrow: q \tens (s \parr t) -> (q \tens s) \parr t;$ and 
$\arrow: (p \parr q) \tens s -> p \parr (q \tens s);$ are 
denoted $\llcad$ and $\rrcad$, respectively.

\begin{rem} \label{rem:upperlowercase}
Let $\K=\Kexplained$ be a \staut\ category and let $\xhom[\K]st$
denote the {external} set of arrows $s \to t$;  
then {\ni}s of the form $\arrow:\perp{p}->\prep{p};$ 
are in bijective correspondence with those of the form  
$\arrow:\xhom[\K]{{p \tens t}}\d->\xhom[\K]{{t \tens p}}\d;$. 

We shall denote this correspondence (summarised below) by a change of
case: 
lower case for {\ni}s of the first form, 
upper case for those of the second.
\begin{diagram}@C=3pc{
  & t \ar[dl]_(.6){\lCurry\omega} \ar[dr]^(.6){\rCurry\psi} 
  && \ar@{}[dr]|-{\iff}
  && p \tens t \ar[d]^-{\omega} 
  \ar@{}[dr]|-{\stackrel{\Cycle{p,t}}\mapsto} 
  & t \tens p \ar[d]^-{\psi}
  \\ \perp{p} \ar[rr]_-{\cycle[p]}
  && \prep{p} 
  &&& \d 
  & \d }
\end{diagram}
\Thatis, given $\cycle$, we set
$\Cycle[p,t](\omega) = \invrCurry{\lCurry\omega\fo\cycle[p]}$; 
and, given $\Cycle$, we set 
$\cycle[p] = \rCurry{\Cycle[p,\perp{p}](\invlCurry{\id[\perp{p}])}}$.

Half of this correspondence can be easily depicted in terms of the
graphical calculus for \staut\ categories developed in \cite{BCST}. 
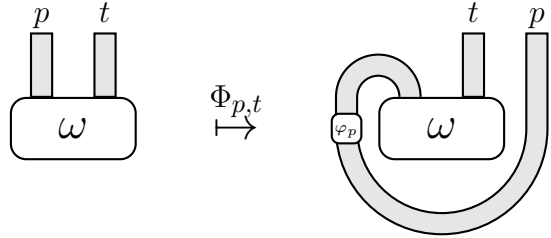
\begin{center}
  \begin{pspicture}[shift=-2](2,0)(14,12)
    \pspolygon[fillstyle=solid,fillcolor=obvcol](10,6)(10,12)(12,12)(12,6)
    \pspolygon[fillstyle=solid,fillcolor=obvcol](4,6)(4,12)(6,12)(6,6)
    \psframe[fillstyle=solid,fillcolor=boxcol](2,0)(14,6)
    \rput(8,3){\Large $\omega$}
    \uput[90](5,12){$\strut p$}
    \uput[90](11,12){$\strut t$}
  \end{pspicture}
  \qquad{\Large$\stackrel{\Cycle[p,t]}{\mapsto}$}\qquad
  \begin{pspicture}[shift=-11](-3,-9)(19,15)
    \pspolygon[fillstyle=solid,fillcolor=obvcol](10,6)(10,12)(12,12)(12,6)
    \uput[90](11,12){$\strut t$}
    \pscustom[fillstyle=solid,fillcolor=obvcol]{ 
      \psarc(2,6){2}{0}{180}
      \psline(0,6)(0,3)(-2,3)(-2,6)
      \psarcn(2,6){4}{180}{0}
    }
    \pscustom[fillstyle=solid,fillcolor=obvcol]{ 
      \psarc(8,3){10}{180}{0}
      \psline(18,3)(18,12)(16,12)(16,3)
      \psarcn(8,3){8}{0}{180}
    }
    \uput[90](17,12){$\strut p$}
    \psframe[fillstyle=solid,fillcolor=boxcol](2,0)(14,6)
    \rput(8,3){\Large $\omega$}
    \psframe[fillstyle=solid,fillcolor=boxcol](-2.5,1.5)(0.5,4.5)
    \rput(-1,3){\tiny $\strut \cycle[p]$}
  \end{pspicture} 
\end{center}
(Our reason for using \emph{ribbons/tape} rather than
\emph{strings/wires} shall soon become apparent.) 

Dually, there is also a bijective correspondence with natural
isomorphisms of the form  
$\arrow:\xhom[\K]\e{{t \parr p}}->\xhom[\K]\e{{p \parr t}};$; 
but we shall have no occasion to use this in the present paper.  
(It will, however, play a prominent role in \cite{EggMcC:cld}.)  
\end{rem}

So, a \cycstaut\ category could be defined either as a pair
$(\K,\cycle)$ with some coherence axioms for $\cycle$, or as a pair  
$(\K,\Cycle)$ with some coherence axioms for $\Cycle$.
Rosenthal \cite{KRo'94} takes the former approach; 
Blute, Lamarche and Ruet \cite{BluLamRue} the latter. 
In section \ref{sec:olds}, we show that these two definitions are
inequivalent;
we further introduce notions of \emph{{\ts}cyclicity} and
\emph{{\ps}cyclicity}, each of which lies between the
stronger notion of cyclicity (that of \cite{BluLamRue}) and the weaker
one (that of \cite{KRo'94}). 
Henceforth, we refer to the weaker notion as \emph{{\kq}cyclicity}, 
so as to reserve the term \emph{cyclicity} for the stronger notion 
(which is also the conjunction of {{\ts}cyclicity} and
{{\ps}cyclicity}).    

Recall that, if $\V$ is a complete and cocomplete, symmetric
and closed monoidal category, and $\c$ is a small $\V$-category, 
then the monoidal category $\C=\Prof[\V]{\c,\c}$ 
is also closed, but not (in general) symmetric\footnote{
Closedness follows from 
\cite{Day}, since it is possible to construct a 
\emph{promonoidal $\V$-category} $\c\op\newt\c$ such that 
\( \V^{\c\op\newt\c}\) and \(\Prof[\V]{\c,\c} \) are equivalent as
monoidal categories.}.
Rosenthal, doubtless inspired by Example \ref{eg:rel}, observed that:  
if $\V$ also admits a dualising object, then so does $\C$;
and, moreover, that the resultant \staut\ structure on $\C$ is always
{\kq}cyclic \cite{KRo'94}. 
We generalise these results in section \ref{sec:news}: 
to derive a \staut\ structure on $\C$ it suffices that $\V$ be
{\ts}cyclic (and $\C$ is also {\ts}cyclic in this case); 
moreover, if $\V$ is also {\kq}cyclic (and therefore cyclic), 
then the same is true of $\C$. 

The reader 
would be quickly forgiven for assuming that a \brdstaut\ category is
necessarily cyclic---after all, one would normally set up the
structure of such a category in such a way as to have 
$\perp{p}=\prep{p}$, 
and it is difficult at first to imagine that an identity {\nt} could
fail to be coherent.   
But the essential import of braidedness is that our graphical calculus 
should no longer be restricted to the plane; 
therefore, we should be able to pull $\omega$ over 
the $p$-ribbon in (our depiction of) $\Cycle[p,t](\omega)$ as follows.  
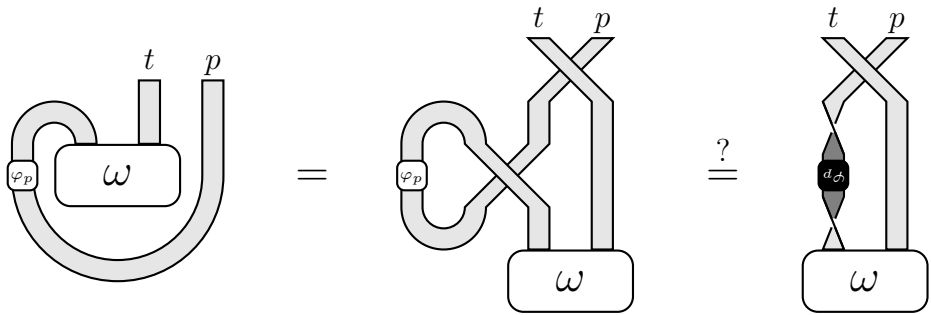
\begin{center}
  \begin{pspicture}[shift=-8](-3,-6)(18,12)
    \pspolygon[fillstyle=solid,fillcolor=obvcol](10,6)(10,12)(12,12)(12,6)
    \uput[90](11,12){$\strut t$}
    \pscustom[fillstyle=solid,fillcolor=obvcol]{ 
      \psline(0,3)(0,6)
      \psarcn(2,6){2}{180}{0}
      \psarc(2,6){4}{0}{180}
      \psline(-2,6)(-2,3)
    }
    \pscustom[fillstyle=solid,fillcolor=obvcol]{ 
      \psarc(8,3){10}{180}{0}
      \psline(18,3)(18,12)(16,12)(16,3)
      \psarcn(8,3){8}{0}{180}
    }
    \uput[90](17,12){$\strut p$}
    \psframe[fillstyle=solid,fillcolor=boxcol](2,0)(14,6)
    \rput(8,3){\Large $\omega$}
    \psframe[fillstyle=solid,fillcolor=boxcol](-2.5,1.5)(0.5,4.5)
    \rput(-1,3){\tiny $\strut \cycle[p]$}
  \end{pspicture}
  \qquad{\Large=}\qquad 
  \begin{pspicture}[shift=-12](-3,-10)(20,19)
    \pscustom[fillstyle=solid,fillcolor=obvcol]{ 
      \psline(0,3)(0,0)
      \psarc(2,0){2}{180}{0}
      \psline(4,0)(10,6)(10,10)(16,16)(18,16)(12,10)(12,6)(6,0)
      \psarcn(2,0){4}{0}{180}
      \psline(-2,0)(-2,3)
    }
    \pscustom[fillstyle=solid,fillcolor=obvcol]{ 
      \psline(0,3)(0,6)
      \psarcn(2,6){2}{180}{0}
      \psline(4,6)(10,0)(10,-4)(12,-4)(12,0)(6,6)
      \psarc(2,6){4}{0}{180}
      \psline(-2,6)(-2,3)
    }
    \psframe[fillstyle=solid,fillcolor=boxcol](-2.5,1.5)(0.5,4.5)
    \rput(-1,3){\tiny $\strut \cycle[p]$}
    \psframe[fillstyle=solid,fillcolor=boxcol](8,-10)(20,-4)
    \rput(14,-7){\Large $\omega$}
    \pspolygon[fillstyle=solid,fillcolor=obvcol](16,-4)(16,10)(10,16)(12,16)(18,10)(18,-4)
    \uput[90](11,16){$\strut t$}
    \uput[90](17,16){$\strut p$}
  \end{pspicture}
  \qquad{\Large$\stackrel?=$}\qquad 
  \begin{pspicture}[shift=-12](8,-10)(20,19)
    \pspolygon[fillstyle=solid,fillcolor=obvcol](10,10)(16,16)(18,16)(12,10)(11,7.5)
    \pspolygon[fillstyle=solid,fillcolor=revcol](10,1)(10,5)(11,7.5)(12,5)(12,1)(11,-1.5)
    \pspolygon[fillstyle=solid,fillcolor=obvcol](10,-4)(12,-4)(11,-1.5)
    \psdots[linecolor=white](11,7.5)(11,-1.5)
    \psline(10,10)(12,5)\psline(10,1)(12,-4)
    \psframe[fillstyle=solid,fillcolor=black](9.5,1.5)(12.5,4.5)
    \rput{180}(11,3){\tiny\white $\cycle[p]$}
    \psframe[fillstyle=solid,fillcolor=boxcol](8,-10)(20,-4)
    \rput(14,-7){\Large $\omega$}
    \pspolygon[fillstyle=solid,fillcolor=obvcol](16,-4)(16,10)(10,16)(12,16)(18,10)(18,-4)
    \uput[90](11,16){$\strut t$}
    \uput[90](17,16){$\strut p$}
  \end{pspicture}
\end{center}
The second (questionable) step is pulling the $p$-ribbon straight.  
(We naturally use \emph{reverse video} to depict the 
{opposite side} of a ribbon or box.)     
Even if $\cycle$ were chosen to be an identity, this step results 
in a $2\pi$-twist on the ribbon $p$---%
so it should not be too surprising to learn that one cannot have a 
cyclicity on a \brdstaut\ category unless it also
carries a \emph{balance} \cite{JoyStr:twists}. 
In fact, this relationship between cyclicity and balancedness is quite 
well-known among people who study
\emph{\boring}\ \staut\ categories
---see, for example, \cite{Yet'92,Mal}.
We show that this relationship does not depend on \boringness\
in section \ref{sec:more};
we further show that even {\kq}cyclicity is not guaranteed 
for a \brdstaut\ category 
by developing a corresponding notion of \emph{{\kq}balance}. 



If the reader wonders why $\cycle[p]$ appears so puny in the
figures above, it is because of the strictification result proven 
in section \ref{sec:zang}.
Let us say that an \arbstaut\ category has \emph{strict negations} 
if all of the de Morgan isomorphisms 
\begin{diagram}@R=1pc@C=2.5pc{
  \perp{p} \tens \perp{q} \ar[r]^-{} 
  & \perp{(q \parr p)} 
  & \perp{(p \tens q)} \ar[r]^-{} 
  & \perp{q} \parr \perp{p} 
  & \e \ar[r]^-{} 
  & \perp{\d} 
  & \perp{\e} \ar[r]^-{} 
  & \d 
  \\ \prep{p} \tens \prep{q} \ar[r]^-{} 
  & \prep{(q \parr p)} 
  & \prep{(p \tens q)} \ar[r]^-{} 
  & \prep{q} \parr \prep{p} 
  & \e \ar[r]^-{} 
  & \prep{\d} 
  & \prep{\e} \ar[r]^-{} 
  & \d }
\end{diagram}
(which are all denoted $\demorgan$) 
and the cancellation isomorphisms
\begin{diagram}{
  p \ar[r]^-{} 
  & \perp{(\prep{p})} 
  & p \ar[r]^-{} 
  & \prep{(\perp{p})} }
\end{diagram} 
(both denoted $\canon$)
are identities, 
and that a \cycstaut\ category has \emph{a strict negation} if, 
in addition, $\cycle$ is the identity \nt. 
Then every \staut\ category is equivalent 
(as a \ld\ category, and therefore also as a \staut\ category) 
to one with strict negations, and every \cycstaut\ category is
equivalent to one with a strict negation.  

We foreshadow the former result 
(which, obviously, does not depend on any of the results of
sections \ref{sec:olds}--\ref{sec:more})  
by suppressing the relevant isomorphisms 
in the graphical calculus for \arbstaut\ categories; 
in effect, we allow strings to be relabelled \scary{on the fly}.  
The latter result entails that the graphical calculus for \cycstaut\
categories should also suppress components of $\cycle$; 
in effect, it should be identical to that for \arbstaut\ categories,
except that a larger number of string relabellings are permitted. 

\section{Coherence axioms} \label{sec:olds}
Throughout this section: 
$\K=\Kexplained$ denotes a \staut\ category; 
$\cycle$, a natural isomorphism $\arrow:\perp{p}->\prep{p};$;
and $\Cycle$ the corresponding natural isomorphism  
$\arrow:\xhom[\K]{{p \tens t}}\d->\xhom[\K]{{t \tens p}}\d;$.
Figure \ref{fig:axioms} lists a number of possible coherence axioms
for $\cycle$; Figure \ref{fig:Axioms} does the same for $\Cycle$. 
(In the latter figure, $\lbind$ and $\rbind$ refer to the functions 
  \begin{diagram}@R=1pc@C=3pc{ 
    \xhom[\K]{(p \parr q)\tens(s \tens t)}\d 
    & \xhom[\K]{p \tens t}\d \times \xhom[\K]{q \tens s}\d 
    \ar[r]^-{}
    \ar[l]_-{}
    & \xhom[\K]{(p \tens q)\tens(s \parr t)}\d }
  \end{diagram}
  which map a pair of arrows $(\omega,\psi)$ 
  to the composites 
  \begin{diagram}@R=1pc@C=8pc{ 
    (p \parr q) \tens (s \tens t) \ar[d]_-{\assoc\inverse} 
      \ar@{}[dddddr]|-{\textrm{and}}
    & (p \tens q) \tens (s \parr t) \ar[d]^-{\assoc} 
    \\ ((p \parr q) \tens s) \tens t \ar[d]_-{\rrcad\tens\id} 
    & p \tens (q \tens (s \parr t)) \ar[d]^-{\id\tens\llcad} 
    \\ (p \parr (q \tens s)) \tens t \ar[d]_-{(\id\parr\psi)\tens\id} 
    & p \tens ((q \tens s) \parr t) \ar[d]^-{\id\tens(\psi\parr\id)}
    \\ (p \parr \d) \tens t \ar[d]_-{\rulaw\tens\id} 
    & p \tens (\d \parr t) \ar[d]^-{\id\tens\lulaw} 
    \\ p \tens t \ar[d]_-{\omega} 
    & p \tens t \ar[d]^-{\omega} 
    \\ \d 
    & \d }
  \end{diagram}
  respectively.) 

In \cite{KRo'94}, a cyclic \staut\ category is defined to be a pair 
$(\K,\cycle)$ such that $\cycle$ satisfies \kimmo.
Rosenthal proves that this axiom suffices to prove that,
for every pair of even integers $m$ and $n$ 
(and every pair of odd integers $m$ and $n$), 
there is a unique isomorphism $p^{*m} \to p^{*n}$ built up from the
components of $\cycle$.

However, in \cite{BluLamRue}, a cyclic \staut\ category is defined to be a
pair $(\K,\Cycle)$ such that $\Cycle$ satisfies \blrzero\ and \blrtwo.
Blute, Lamarche and Ruet prove: 
that \blrzero\ is equivalent to \tnul;  
that \kimmo\ is equivalent to \kprime;  
and, that the latter follows from \blrtwo.
They further conjecture that their definition is strictly stronger
than that of Rosenthal.

\begin{eg}
  Consider $(\Vecfd,\vnewt,k)$, where $k$ is a field not of
  characteristic two;   
  since $\vnewt$ is symmetric, we can (and do) choose to define 
  $\perp\blank$ and $\prep\blank$ so that $\perp{p}=\prep{p}$ for 
  all spaces $p$.  
  So each non-zero scalar determines a natural isomorphism 
  $\arrow:\perp{p}->\prep{p};$;  
  the latter satisfies \kimmo\ if and only if the scalar is $\pm1$, 
  and \tnul\ if and only if the scalar is 1.  
\end{eg}


We shall find it convenient to 
consider yet more possible (combinations of) axioms, as follows.

\begin{defs} \label{defs:main}
  We call $\cycle$: a
  \emph{\ps\cthing}, if it satisfies \pbin; 
  \emph{\kq\cthing}, if it satisfies \kimmo;
  \emph{\ts\cthing}, if it satisfies \tbin; 
  \emph{\cthing}, if it is both a \ts\cthing\ and a \ps\cthing.

  The pair $(\K,\cycle)$ is called a 
  \emph{(\ps-, \kq-, \ts-)cyclic \staut\ category} 
  whenever $\cycle$ is a (\ps-, \kq-, \ts-)cycle.  
\end{defs}

\begin{figure}
  \begin{diagram}@C=3pc@L=0pt{
      \perp\d \ar[rr]^-{\strut\cycle[\d]} \ar[dr]_-{\demorgan} 
      & \ar@{}[d]|(.4){\pnul}
      & \prep\d \ar[dl]^-{\demorgan} 
      & r \ar[dl]_-{\canon[r]} \ar@{}[d]|(.6){\kimmo} \ar[dr]^-{\canon[r]} 
      & \perp\e \ar[rr]^-{\strut\cycle[\e]} \ar[dr]_-{\demorgan} 
      & \ar@{}[d]|(.4){\tnul}
      & \prep\e \ar[dl]^-{\demorgan} 
      \\ & \e 
      & \perp{(\prep{r})} \ar[r]_-{\strut\perp{(\cycle[r])}}
      & \perp{(\perp{r})} \ar[r]_-{\strut\cycle[\perp{r}]} 
      & \prep{(\perp{r})} 
      & \d }
  \end{diagram}
  \begin{diagram}@C=2pc{
      \perp{(p \parr q)} \ar[rr]^-{\cycle[p \parr q]} \ar[d]_-{\demorgan} 
      &&  \prep{(p \parr q)} \ar[d]^-{\demorgan} \ar@{}[dll]|-{\pbin}
      &  \perp{(p \tens q)} \ar[rr]^-{\cycle[p \tens q]} \ar[d]_-{\demorgan} 
      &&  \prep{(p \tens q)} \ar[d]^-{\demorgan} \ar@{}[dll]|-{\tbin}
      \\ \perp{q} \tens \perp{p} \ar[rr]_-{\cycle[q] \tens \cycle[p]}
      &&  \prep{q} \tens \prep{p} 
      &  \perp{q} \parr \perp{p} \ar[rr]_-{\cycle[q] \parr \cycle[p]}
      &&  \prep{q} \parr \prep{p} }
  \end{diagram}
  \caption{Axioms for a $\cycle$} \label{fig:axioms}
\end{figure}

\begin{lem} \label{lem:mbm}
  The following dependencies exist between the axioms listed in
  Figure \ref{fig:axioms}.  
\[\begin{array}{rclcrcl}
    \tbin &\entails& \tnul &\hspace{2cm}& 
    \tbin &\entails& \evan{\pnul \iff \kimmo} \\
    \pbin &\entails& \pnul && 
    \pbin &\entails& \evan{\tnul \iff \kimmo} \\
    \kimmo &\entails& \evan{\tnul \iff \pnul} && 
    \kimmo &\entails& \evan{\tbin \iff \pbin} 
\end{array}\]
In particular, a \cthing\ is a \kq\cthing; 
moreover, 
each of the following four pairs of axioms is equivalent to
cyclicity. 
  \[ \classof{\pnul,\tbin} \qquad 
  \classof{\kimmo,\tbin} \qquad 
  \classof{\pbin,\kimmo} \qquad 
  \classof{\pbin,\tnul} \]
\end{lem}
\begin{proof}
  We discuss only the first row of assertions: 
  the second row of assertions are dual; 
  and, although the third row of assertions 
  (which are, in any case, much less surprising) 
  can be proven directly, it is simpler to see them as a 
  corollary of Lemma \ref{lem:jme} and its dual.  

That $\tbin\entails\tnul$ follows from a more general fact---namely, 
that a \scary{semigroupal} natural isomorphism between strong monoidal
functors is necessarily monoidal
(see Lemma \ref{appx:snis}). 
It should be clear that 
\tbin\ asserts that $\cycle$ is semigroupal, and that $\tbin$ together with $\tnul$ asserts that $\cycle$ is monoidal,
\wrt\ strong monoidal functors
$(\K,{\tens},\e)\to(\K,{\parr},\d)$ overlying $\perp\blank$ and
$\prep\blank$.)

To prove $\tbin\entails\evan{\pnul\iff\kimmo}$, 
it will be convenient to derive a form of $\tbin$
which does not explicitly refer to $\parr$.
This is achieved by applying the natural isomorphisms 
$(x \loll z) \isoto \perp{x}\parr{z}$ and 
$(z \llol y) \isoto {z}\parr\prep{y}$, as follows.

\begin{diagram}@C=0pc{
    \perp{(p \tens q)} 
    \ar[rr]_-{\demorgan} \ar[dr]|-{\sim}
    \fixwonky{u}{rrrrrrrr}{^-{\cycle[p \tens q]}}
    && \perp{q} \parr \perp{p} 
    \ar[rr]_-{\id[q] \parr \cycle[p]} \ar[dl]|-{\sim}
    && \perp{q} \parr \prep{p} 
    \ar[rr]_-{\cycle[q] \parr \id[p]} \ar[dl]|-{\sim} \ar[dr]|-{\sim}
    && \prep{q} \parr \prep{p} 
    \ar[rr]_-{\demorgan} \ar[dr]|-{\sim}
    && \prep{(p \tens q)} 
    \\ & q \loll (\perp{p}) \ar[rr]_-{\id[q]\loll\cycle[p]} 
    &&  q \loll (\prep{p}) \ar[rr]_-{\sim} 
    && (\perp{q}) \llol p \ar[rr]_-{\cycle[q]\llol\id[p]}
    && (\prep{q}) \llol p \ar[ur]|-{\sim} }
\end{diagram}

The outer hexagon of the diagram above forms the central cell in the
diagram below,  
(But we have replaced $\psi\loll\omega$ and $\omega\llol\psi$
by their more colloquial forms:
$\psi\of\blank\of\omega$ and $\omega\fo\blank\fo\psi$, respectively.)  
The two cells labelled \naty\ are naturality squares, 
and all the unlabelled cells (including the outermost square)
are 
tautologies
that hold in arbitrary \staut\ categories.  

\begin{diagram}@C=4pc{ 
    \e \ar@{=}[rr] \ar@{=}[ddd]
    && \e \ar[dr]^-{\demorgan} \ar[dl]_-{\demorgan} \ar@{=}[rr]
    \ar@{}[d]|-{\qpnul}   
    && \e \ar[ddd]^{\name{\id[p]}} 
    \\ & \perp\d \ar[rr]|-{\cycle[\d]} \ar[d]_-{\perp\gamma} 
    & \ar@{}[d]|-{\naty}  
    & \prep\d \ar[d]^-{\prep\gamma} 
    \\ & \perp{(p\tens\perp{p})} \ar[rr]|-{\cycle[p\tens\perp{p}]} 
    \ar[d]_-{\sim} \ar@{}[ddrr]|-{\vartbin}
    && \prep{(p\tens\perp{p})} & 
    \\ \e \ar[dr]|-{\name{\cycle[p]}}
    \ar[d]_-{\name{\id[\prep{p}]}} 
    \ar[r]^-{\name{\id[\perp{p}]}}
    & \perp{p}\loll\perp{p} 
    \ar[d]^-{\blank\fo\cycle[p]} 
    && \prep{(\perp{p})}\llol p \ar[u]_-{\sim} 
    & p \llol p \ar[l]_-{\canon[p]\of\blank} 
    \ar[d]^-{\canon[p]\of\blank} \ar@{}[dl]|-{\qkimmo}
    \\ \prep{p}\loll\prep{p} 
    \ar[r]_-{\cycle[p]\fo\blank}
    \fixwonky{d}{rrrr}{_-{\sim}}
    & \perp{p}\loll\prep{p} \ar[rr]^-{\sim}_-{\naty}
    && \perp{(\perp{p})}\llol p 
    \ar[u]^-{\cycle[\perp{p}]\of\blank} 
    & \perp{(\prep{p})}\llol p 
    \ar[l]^-{\perp{(\cycle[p])}\of\blank} 
  }   
\end{diagram}

The cell labelled $\qpnul$ equals $\pnul$.
If it holds, then we conclude that 
\[ \name{\canon[p]}=\canon[p]\of(\name{\id[p]})
   =\cycle[\perp{p}]\of(\perp{(\cycle[p])}\of(\canon[p]\of(\name{\id[p]})))
   =\name{\cycle[\perp{p}]\of\perp{(\cycle[p])}\of\canon[p]} \]
---and this is tautologously equivalent to \kimmo.

Conversely, if \kimmo\ holds, then so does $\qkimmo$; 
if we choose $p$ to be either $\d$ or $\e$, then all the arrows 
(including the counit $\gamma$) are invertible, 
and this allows us to conclude that $\qpnul$ (which is $\pnul$) holds. 
\end{proof}

The main result of this section is that our definition of cyclic 
\staut\ category agrees with that of \cite{BluLamRue}; 
this is an immediate corollary of the following lemma.  


\begin{figure}
  \begin{diagram}@C=4pc{ 
      \xhom[\K]{(p \tens q) \tens t}\d 
      \ar[d]_-{\assoc[p,q,t]\inverse\fo\blank} 
      \ar[r]^-{\Cycle[p \tens q,t]} 
      \ar@{}[ddr]|-{\blrtwo}
      & \xhom[\K]{t \tens (p \tens q)}\d 
      \ar[d]^-{\assoc[t,p,q]\fo\blank} 
      & \xhom[\K]{t \tens p}\d 
      \ar[dd]^-{\Cycle[t,p]=\Cycle[p,t]\inverse}_-{\kprime}
      \\ \xhom[\K]{p \tens (q \tens t)}\d 
      \ar[d]_-{\Cycle[p,q \tens t]} 
      & \xhom[\K]{(t \tens p) \tens q}\d 
      \ar[d]^-{\Cycle[t \tens p,q]}
      \\ \xhom[\K]{(q \tens t) \tens p}\d 
      \ar[r]_-{\assoc[q,t,p]\inverse\fo\blank} 
      & \xhom[\K]{q \tens (t \tens p)}\d 
      & \xhom[\K]{p \tens t}\d }
  \end{diagram}
  \begin{diagram}@C=4pc{ 
      \xhom[\K]{(p \tens q) \tens t}\d 
      \ar[d]_-{\assoc[p,q,t]\inverse\fo\blank} 
      \ar[r]^-{\Cycle[p \tens q,t]} 
      \ar@{}[ddr]|-{\etwo}
      & \xhom[\K]{t \tens (p \tens q)}\d 
      & \xhom[\K]{\e \tens t}\d 
      \ar[dd]_-{\Cycle[\e,t]}
      \ar[dr]^-{\lulaw[t]\inverse\fo\blank}
      \\ \xhom[\K]{p \tens (q \tens t)}\d 
      \ar[d]_-{\Cycle[p,q \tens t]} 
      & \xhom[\K]{(t \tens p) \tens q}\d 
      \ar[u]_-{\assoc[t,p,q]\inverse\fo\blank} 
      && \xhom[\K]t\d \ar[dl]^-{\rulaw[t]\fo\blank}          
      \ar@{}[l]|(.67){\blrzero}
      \\ \xhom[\K]{(q \tens t) \tens p}\d 
      \ar[r]_-{\assoc[q,t,p]\inverse\fo\blank} 
      & \xhom[\K]{q \tens (t \tens p)}\d \ar[u]_-{\Cycle[q,t \tens p]}
      & \xhom[\K]{t \tens \e}\d }
  \end{diagram}
  \begin{diagram}@C=4pc{ 
      \xhom[\K]{p \tens (s \tens t)}\d 
      \ar[r]^-{\Cycle[p,s \tens t]} 
      \ar[d]_-{\assoc[p,s,t]\fo\blank} 
      \ar@{}[ddr]|-{\eagain}
      & \xhom[\K]{(s \tens t) \tens p}\d 
      & \xhom[\K]{t \tens \e}\d \ar[dd]_-{\Cycle[t,\e]}
      \ar[dr]^-{\rulaw[t]\inverse\fo\blank}
      \\ \xhom[\K]{(p \tens s) \tens t}\d 
      \ar[d]_-{\Cycle[p \tens s,t]}
      & \xhom[\K]{s \tens (t \tens p)}\d 
      \ar[u]_-{\assoc[s,t,p]\fo\blank} 
      && \xhom[\K]t\d \ar[dl]^-{\lulaw[t]\fo\blank} 
      \ar@{}[l]|(.67){\mzero}      
      \\ \xhom[\K]{t \tens (p \tens s)}\d 
      \ar[r]_-{\assoc[t,p,s]\fo\blank} 
      & \xhom[\K]{(t \tens p) \tens s}\d 
      \ar[u]_-{\Cycle[t \tens p,s]} 
      & \xhom[\K]{\e \tens t}\d 
      }
  \end{diagram}
  \begin{diagram}{
      \xhom[\K]{p \tens t}\d \times \xhom[\K]{q \tens s}\d 
      \ar[r]^-{\Cycle[p,t]\times\Cycle[q,s]}
      \ar[dd]_-{\lbind}
      & \xhom[\K]{t \tens p}\d \times \xhom[\K]{s \tens q}\d 
      \ar[d]^-{\sim}
      \\ \ar@{}[r]|-{\mtwo}
      & \xhom[\K]{s \tens q}\d \times \xhom[\K]{t \tens p}\d 
      \ar[d]^-{\rbind}
      \\ \xhom[\K]{(p \parr q)\tens(s \tens t)}\d 
      \ar[r]_-{\Cycle[p \parr q,s \tens t]}
      & \xhom[\K]{(s \tens t)\tens(p \parr q)}\d }
  \end{diagram}
  \begin{diagram}{
      \xhom[\K]{p \tens t}\d \times \xhom[\K]{q \tens s}\d 
      \ar[r]^-{\Cycle[p,t]\times\Cycle[q,s]}
      \ar[dd]_-{\rbind}
      & \xhom[\K]{t \tens p}\d \times \xhom[\K]{s \tens q}\d 
      \ar[d]^-{\sim}
      \\ \ar@{}[r]|-{\varmtwo}
      & \xhom[\K]{s \tens q}\d \times \xhom[\K]{t \tens p}\d 
      \ar[d]^-{\lbind}
      \\ \xhom[\K]{(p \tens q)\tens(s \parr t)}\d 
      \ar[r]_-{\Cycle[p \tens q,s \parr t]}
      & \xhom[\K]{(s \parr t)\tens(p \tens q)}\d }
  \end{diagram}
  \caption{Axioms for a $\Cycle$}\label{fig:Axioms}
\end{figure}

\begin{lem} \label{lem:jme}
  The following axioms are equivalent: \tbin, \etwo, \varmtwo. 
  Moreover, 
  $\blrtwo$ is equivalent to ${\kprime\meet\etwo}$. 
\end{lem}
\begin{proof}
Using the $\parr$-free version of \tbin\ derived in the proof of
Lemma \ref{lem:mbm} above, we 
{chase the two diagrams} simultaneously: 
given $\omega\in\xhom[\K]{(p \tens q) \tens t}\d$, 
\begin{diagram}@C=2pt{ 
  &&&&& t \ar@/_2ex/[dlllll]+U_-{\lCurry\omega} 
  \ar@{.>}@/_/[dlll]+U \ar@{.>}[dl]+U \ar@{.>}[dr]+U \ar@{.>}@/^/[drrr]+U
  \ar@{.>}@/^2ex/[drrrrr]+U
  \\ \perp{(p \tens q)} \ar[rr]_-{\sim} 
  && q \loll (\perp{p}) \ar[rr]_-{\strut\id[q]\loll\cycle[p]} 
  &&  q \loll (\prep{p}) \ar[rr]_-{\sim} 
  && (\perp{q}) \llol p \ar[rr]_-{\strut\cycle[q]\llol\id[p]}
  &&  (\prep{q}) \llol p  \ar[rr]_-{\sim}
  && \prep{(p \tens q)} 
  \\ (p \tens q) \tens t \ar[d]_-{\omega} 
     \ar@{}[drr]|-{\stackrel{\assoc[p,q,t]\inverse\fo\blank}\mapsto}
  && p \tens (q \tens t) \ar[d]^-{}
     \ar@{}[drr]|-{\stackrel{\Cycle[p,q \tens t]}\mapsto}
  && (q \tens t) \tens p \ar[d]^-{} 
     \ar@{}[drr]|-{\stackrel{\assoc[q,t,p]\inverse\fo\blank}\mapsto}
  && q \tens (t \tens p) \ar[d]^-{} 
     \ar@{}[drr]|-{\stackrel{\Cycle[q,t \tens p]}\mapsto}
  && (t \tens p) \tens q \ar[d]^-{} 
     \ar@{}[drr]|-{\stackrel{\assoc[t,p,q]\inverse\fo\blank}\mapsto}
  && t \tens (p \tens q) \ar[d]^-{} 
  \\ \d 
  && \d
  && \d 
  && \d
  && \d
  && \d }
\end{diagram}  
should equal 
\begin{diagram}{ 
  & t \ar[dl]+U_-{\lCurry\omega} \ar@{.>}[dr]+U 
  \\ \perp{(p \tens q)} \ar[rr]_-{\cycle[p \tens q]} 
  && \prep{(p \tens q)} 
  \\ (p \tens q) \tens t \ar[d]_-{\omega} 
     \ar@{}[drr]|-{\stackrel{\Cycle[p \tens q,t]}\mapsto}
  && t \tens (p \tens q) \ar[d]^-{} 
  \\ \d 
  && \d }
\end{diagram}
respectively. 

Now, if \etwo\ holds, then \varmtwo\ 
can be verified by chasing a pair $(\omega,\psi)$ through the
diagram, as follows.  
\begin{eqnarray} 
\lefteqn{\Cycle[p \tens q,s \parr t](\rbind(\omega,\psi))} \nonumber
\\ &=& 
{\assoc[s \parr t,p,q]\inverse\fo\Cycle[q,(s \parr t) \tens
      p](\assoc[q,s \parr t,p]\inverse\fo\Cycle[p,q \tens (s \parr
      t)](\assoc[p,q,s \parr t]\inverse\fo\rbind(\omega,\psi)))}
\\ &=& \assoc[s \parr t,p,q]\inverse\fo\Cycle[q,(s \parr t) \tens
      p](\assoc[q,s \parr t,p]\inverse\fo\Cycle[p,q \tens (s \parr
      t)](\id[p]\tens(\llcad[q,s,t]\fo\psi\parr\id[t]\fo\lulaw[t])
      \fo\omega))
\\ &=& \assoc[s \parr t,p,q]\inverse\fo\Cycle[q,(s \parr t) \tens
      p](\assoc[q,s \parr t,p]\inverse\fo
      (\llcad[q,s,t]\fo\psi\parr\id[t]\fo\lulaw[t])\tens\id[p]\fo
      \Cycle[p,t](\omega))
\\ &=& \assoc[s \parr t,p,q]\inverse\fo\Cycle[q,(s \parr t) \tens
      p](\id[q]\tens(\rrcad[s,t,p]\fo\id[s]\parr\Cycle[p,t](\omega)
      \fo\rulaw[s])\fo\psi) 
\\ &=& \assoc[s \parr t,p,q]\inverse\fo
   (\rrcad[s,t,p]\fo\id[s]\parr\Cycle[p,t](\omega)\fo\rulaw[s])
   \tens\id[q]\fo\Cycle[q,s](\psi) 
\\ &=& \lbind(\Cycle[q,s](\psi),\Cycle[p,t](\omega))
\end{eqnarray}
(Equation (1) applies \etwo; 
Equations (2) and (6), the definitions of $\rbind$ and $\lbind$,  
respectively;
Equations (3) and (5), the naturality of $\Cycle$.  
Equation (4) is a simple exercise in \ld\ category theory---see
Lemma \ref{appx:base}.)

Conversely, we can derive \tbin\ from \varmtwo, by applying the latter
to the pair $\omega=\gamma_p=\invlCurry{\id[\perp{p}]}$ and 
$\psi=\gamma_q=\invlCurry{\id[\perp{q}]}$. 
(Note that $\rbind(\gamma_p,\gamma_q)\iso\gamma_{p \tens q}$, 
as arrows, 
via the de Morgan isomorphism 
$\demorgan : \arrow:\perp{(p \tens q)}->\perp{q}\parr\perp{p};$.)

Finally, it is clear that $\evan{\kprime\meet\etwo}\entails\blrtwo$ and 
$\evan{\blrtwo\meet\kprime}\entails\etwo$.
But since $\blrtwo\entails\kprime$ the latter can be sharpened to 
$\blrtwo\entails\evan{\etwo\meet\kprime}$. 
\end{proof}

For the sake of completeness, we note without proof that 
the axioms \pnul\ and \mzero\ are equivalent, 
as are \pbin, \mtwo\ and \eagain. 
Note that \mzero\ is \blrzero-for-$\Cycle\inverse$, 
\eagain\ {is} \etwo-for-$\Cycle\inverse$,  
and \varmtwo\ is \mtwo-for-$\Cycle\inverse$.
Hence, $\kprime\entails\evan{\mzero\iff\blrzero}$ and 
$\kprime\entails\evan{\mtwo\iff\varmtwo}$ are trivial.  

\section{Enriched profunctors and cyclicity} \label{sec:news}
Throughout this section: $\V=\Vexplained$ denotes a {\ts}cyclic 
\staut\ category; and, when we speak of $\V$-categories
($\V$-profunctors, \etc), then we mean $\tV$-categories 
(resp., $\tV$-profunctors, \etc).

\begin{thm} \label{thm:main}
  Let $\c$ be a small $\V$-category; 
  then $\Prof[\V]{\c,\c}$ is {\ts}cyclic \staut. 
  Moreover, if $\V$ is also {\kq}cyclic (and therefore cyclic), 
  then the same is true of $\Prof[\V]{\c,\c}$. 
\end{thm}

Before proceeding with the proof, 
we discuss a few of the issues that arise in the consideration of  
non-symmetric $\V$.    

\begin{rem} \label{rem:varia}
Given an arbitrary (\ts)cyclic \staut\ category $\V$, 
it is impossible to define the product $\a\newt\b$ of
$\V$-categories $\a$ and $\b$; 
this requires at least a braiding on $\V$.  
Similarly, it is impossible to define the opposite $\c\op$ of a
$\V$-category $\c$; 
this requires at least 
a braiding or an \emph{involution} in the sense of
\cite{Egg:imc}.  

Thus, the notion of $\V$-profunctor $\proarrow:\a->\b;$ must be defined 
in more elementary terms than the customary 
\scary{$\V$-functor $\a\op\newt\b\to\V$}. 
This is done in \cite{Benabou}, and also in \cite{Str'83}: 
very simply, a $\V$-profunctor $f : \proarrow:\a->\b;$ is an 
$(\ob\a\times\ob\b)$-indexed family of $\V$-objects, $\xhom[f]qr$, 
together with 
$\V$-arrows 
\begin{diagram}@C=4pc{
  \xhom[\a]pq\tens\xhom[f]qr \ar[r]^-{\lact[p,q,r]} & \xhom[f]pr 
  & \xhom[f]qr\tens\xhom[\b]rs \ar[r]^-{\ract[q,r,s]} & \xhom[f]qs } 
\end{diagram}
(for all $p,q\in \ob\a$ and $r,s\in \ob\b$) 
satisfying the five obvious associativity and unitality axioms.

[It is perhaps helpful to imagine $\xhom[f]qr$ as consisting of 
  \emph{oblique} arrows $q \to r$, and $\lact[p,q,r]$ and
  $\ract[q,r,s]$ as performing the composition of these with 
  \emph{genuine} arrows $p \to q$ and $r \to s$ in $\a$ and $\b$
  respectively. 
  For example, the identity profunctor on $\c$ is obtained by
  regarding (the object of) genuine arrows in $\c$ as (an object of) 
  oblique arrows.]   

A \emph{modulation} of $\V$-profunctors $\omega : f\To g$ is a
family of $\V$-arrows $\xhom[\omega]qr : \arrow:\xhom[f]qr->\xhom[g]qr;$ 
which are suitably compatible with the multiplicative structure of $f$
and $g$.   
(We borrow the term modulation from \cite{CKSW}.)

Composition of $\V$-profunctors is a routine application of
\emph{coends}: 
given profunctors $f : \proarrow:\a->\b;$ and $g : \proarrow:\b->\c;$, 
we take the family 
\[ \xhom[f \tens g]qs := \int^r \xhom[f]qr\tens\xhom[g]rs, \]
together with a left $\a$-action derived from that of $f$, 
and a right $\c$-action derived from that of $g$;
see \cite{Benabou} or \cite{Str'83} for details.  

Using these more elementary definitions, it is not clear that 
$\Prof[\V]{\c,\c}$ should be closed;
indeed, it appears to us to be untrue in full generality.  
Certainly, we have (so far) been unable to deduce a \staut\ structure
on $\Prof[\V]{\c,\c}$ when $\V$ is an arbitrary \staut\ category: 
it seems that the hypothesis of {\ts\cprop} cannot be weakened further. 
This asymmetry (that {\ts\cprop} is essential and {\ps\cprop} optional)
stems partly from the asymmetry contained in the very definitions of
$\V$-category and $\V$-profunctor (which are cast in terms of $\tens$
and not $\parr$), and partly from the
inherent \scary{two-dimensionality} of the notion of $\V$-matrix, 
which underlies that of $\V$-profunctor.  
\end{rem}


\begin{proof} 
  Let 
  $\a$ and $\b$ be small $\V$-categories, 
  and $f$ a $\V$-profunctor $\proarrow:\a->\b;$.
  Then 
  we use the \scary{contraposition} isomorphisms
  \( x \loll y \to \perp{x} \llol \perp{y} \)
and 
  \( z \llol x \to \prep{z} \loll \prep{x} \) in $\V$
  to construct $\V$-arrows 
  \[ \vcenter{
    \infer{\xhom[\b]pq\tens\perp{(\xhom[f]rq)}\to\perp{(\xhom[f]rp)}}{
      \infer{\xhom[\b]pq\to\perp{(\xhom[f]rp)}\llol\perp{(\xhom[f]rq)}}{
        \infer{\xhom[\b]pq\to\xhom[f]rp\loll\xhom[f]rq}{
          \xhom[f]rp\tens\xhom[\b]pq\to\xhom[f]rq }}}} 
  \qquad
  \vcenter{
    \infer{\prep{(\xhom[f]rq)}\tens\xhom[\a]rs\to\prep{(\xhom[f]sq)}}{
      \infer{\xhom[\a]rs\to\prep{(\xhom[f]rq)}\loll\prep{(\xhom[f]sq)}}{
        \infer{\xhom[\a]rs\to\xhom[f]rq\llol\xhom[f]sq}{
          \xhom[\a]rs\tens\xhom[f]sq\to\xhom[f]rq }}}} 
  \]
  which exhibit: a left action of $\b$ on the family of objects 
  \( \xhom[\perp{f}]qr:=\perp{(\xhom[f]rq)}; \)
  and, a right action of $\a$ on the family of objects 
  \( \xhom[\prep{f}]qr:=\prep{(\xhom[f]rq)}. \)
  In other words, we obtain profunctors $\perp{f} : \proarrow:\b->\one;$
  and $\prep{f} : \proarrow:\one->\a;$. 

  At this point it is natural to use $\cycle$ to transport
  the left action of $\b$ on $\perp{f}$ to one on $\prep{f}$
  and to transport the right action of $\a$ on $\prep{f}$ to one on $\perp{f}$, as follows:
  \begin{diagram}@C=4pc@R=1pc{
    \xhom[\b]pq\tens\prep{(\xhom[f]rq)}
    \ar[r]^-{\id\tens\cycle\inverse} 
    & \xhom[\b]pq\tens\perp{(\xhom[f]rq)}
    \ar[r]^-{} 
    & \perp{(\xhom[f]rp)}
    \ar[r]^-{\cycle} 
    & \prep{(\xhom[f]rp)}
    \\ \perp{(\xhom[f]rq)}\tens\xhom[\a]rs
    \ar[r]^-{\cycle\tens\id} 
    & \prep{(\xhom[f]rq)}\tens\xhom[\a]rs
    \ar[r]^-{} 
    & \prep{(\xhom[f]sq)}
    \ar[r]^-{\cycle\inverse} 
    & \perp{(\xhom[f]sq)} }
  \end{diagram}
  To show that we obtain profunctors $\perp{f} : \proarrow:\b->\a;$ and
  $\prep{f} : \proarrow:\b->\a;$ from the four actions described above, 
  it remains to show that middle associativity holds; 
  this is surprisingly difficult.

  An alternative approach, favoured by the second author, is to
  derive maps
  $\arrow:\xhom[\b]pq\tens\prep{(\xhom[f]rq)}->\prep{(\xhom[f]rp)};$  
  and $\arrow:\perp{(\xhom[f]rq)}\tens\xhom[\a]rs->\perp{(\xhom[f]sq)};$
  by deCurrying each of the following composites.  
  \begin{diagram}@C=4pc@R=1pc{ 
      & \prep{\xhom[\b]pq} \parr \prep{(\xhom[f]rp)}
      \ar[r]^-{\cycle\parr\id}
      & \perp{\xhom[\b]pq} \parr \prep{(\xhom[f]rp)}
      \ar[d]|-\sim 
      \\ \prep{(\xhom[f]rq)}
      \ar[r]^-{}
      & \prep{(\xhom[f]rp\tens\xhom[\b]pq)}
      \ar[u]|-{\sim} 
      & \xhom[\b]pq\loll \prep{(\xhom[f]rp)}
      \\ \perp{(\xhom[f]rq)} \ar[r]^-{}
      & \perp{(\xhom[\a]rs \tens \xhom[f]sq)}
      \ar[d]|-\sim 
      & \perp{(\xhom[f]sq)} \llol \xhom[\a]rs
      \\
      & \perp{(\xhom[f]sq)} \parr \perp{\xhom[\a]rs}
      \ar[r]^-{\cycle\parr\id}
      & \perp{(\xhom[f]sq)} \parr \prep{\xhom[\a]rs}
      \ar[u]|-{\sim} }
  \end{diagram}
  The advantage of this approach is that middle associativity becomes
  trivial; the disadvantage is that ordinary (left- and right-)
  associativity becomes difficult. 

  The solution is to show that both approaches result in the same
  pair of arrows.  
  This is a simple exercise, yet it relies crucially on axiom \tbin:
  see Lemma \ref{appx:strings}.    


  By contrast, it is essentially tautologous that the family of
  $\V$-arrows 
  \begin{diagram}@C=7pc{
      \xhom[\perp{f}]qr=\perp{(\xhom[f]rq)} 
      \ar[r]^-{\cycle[{\xhom[f]rq}]}
      & \prep{(\xhom[f]rq)}=\xhom[\prep{f}]qr }
  \end{diagram}
  defines an invertible modulation of profunctors
  $\widetilde{\cycle}_f: \arrow:\perp{f}->;\prep{f}$.  


  Now we define $\parr$ and $\d$ as the de Morgan duals of $\tens$ and
  $\e$; 
  it is easy to work out that
  $f \parr g$ can be equivalently defined using \emph{ends}:   
  \[ \xhom[f \parr g]qs := \int_r \xhom[f]qr\parr\xhom[g]rs, \]
  together with a left $\a$-action derived from that of $f$, 
  and a right $\c$-action derived from that of $g$.
  Constructing the necessary linear distribution is routine:
  a modulation 
  $\arrow:f \tens (g \parr h)->(f \tens~g) \parr h;$
  is uniquely determined by an appropriate family of arrows 
  \begin{inline}{ 
    \xhom[f]pq \tens \xhom[g \parr h]qs 
    \ar[r]^-{} 
    & \xhom[f \tens g]pr \parr \xhom[h]rs }
  \end{inline}, 
  and it is neither hard to see that 
  \begin{diagram}{ 
    \xhom[f]pq \tens \xhom[g \parr h]qs 
    \ar[d]_-{\id\tens\pi_r} 
    & \xhom[f \tens g]pr \parr \xhom[h]rs 
    \\ \xhom[f]pq \tens (\xhom[g]qr \parr \xhom[h]rs) 
    \ar[r]^-{\llcad} 
    & (\xhom[f]pq \tens \xhom[g]qr) \parr \xhom[h]rs 
    \ar[u]_-{\copi_q\parr\id} }
  \end{diagram}        
  is such a family of arrows, nor that the resultant modulations
  satisfy the necessary coherence axioms 
  (compare with \cite{Egg:staf}). 
  Thus $\Profsymbol[\V]$ is a linear bicategory;
  in particular, $\Prof[\V]{\c,\c}$ is a \ld\ category for every $\c$. 

  The construction of modulations \[ \arrow\tau:\e->\perp{f}\parr f;, \quad
  \arrow\tau:\e->f\parr\prep{f};, \quad \arrow\gamma:f\tens\perp{f}->\d;\quad
  \textrm{and} \quad \arrow\gamma:\prep{f}\tens f->\d; \] satisfying the necessary 
  (linear) triangle identities (thus proving $\perp{f}$ and $\prep{f}$
  to be, respectively, right and left duals of $f$) is similarly
  routine.   
  Thus $\Profsymbol[\V]$ is a $*$-linear bicategory;
  in particular, $\Prof[\V]{\c,\c}$ 
  is a \staut\ category.  

  Given a modulation $\omega : \arrow:f->g;$, the dual modulation 
  $\perp\omega: \arrow:\perp{g}->\perp{f};$ is calculated pointwise, 
  as one would expect.
  \begin{diagram}@R=1pc{
    \xhom[\perp{g}]pq \ar[r]^-{\xhom[\perp\omega]pq} \ar@{=}[d]
    & \xhom[\perp{f}]pq \ar@{=}[d]
    \\ \perp{(\xhom[g]qp)} \ar[r]_-{\perp{(\xhom[\omega]qp)}} 
    & \perp{(\xhom[f]qp})}
  \end{diagram}
  Moreover, the cancellation modulations
  $\arrow:f->\perp{(\prep{f})};$ and $\arrow:f->\perp{(\prep{f})};$, 
  can also be calculated pointwise; 
  it follows that the de Morgan modulations, such as 
  $\arrow:\perp{(f \tens g)}->\perp{g}\parr\perp{f};$,
  are related to those of $\V$ as follows. 
  \begin{diagram}@R=1pc{
    \xhom[\perp{f \tens g}]pr \ar[r]^-{\demorgan} 
    \ar@{=}[d]
    & \xhom[\perp{g}\parr\perp{f}]pr \ar[r]^-{\pi_q} 
    & \xhom[\perp{g}]pq\parr\xhom[\perp{f}]qr               
    \ar@{=}[d]
    \\ \perp{(\xhom[f \tens g]rp)} \ar[r]_-{\perp{\copi_q}} 
    & \perp{(\xhom[f]rq\tens\xhom[g]qp)} \ar[r]_-{\demorgan} 
    & \perp{(\xhom[g]qp)} \parr \perp{(\xhom[f]rq})}
  \end{diagram}
  These observations allow one to quickly conclude that the
  {\ts}cyclicity of $\V$ is inherited by $\Prof[\V]{\c,\c}$, and also the
  {\kq}cyclicity, if $\V$ enjoys that property.
\end{proof}

\section{Braidings and cyclicities} \label{sec:more}
Throughout this section: 
$\K=\Kexplained$ denotes a \brdstaut\ category; 
$\balance$, a natural isomorphism of $\arrow:\Id[\K]->\Id[\K];$; 
$\cycle$, a natural isomorphism $\arrow:\perp\blank->\prep\blank;$;
and $\Cycle$ the corresponding natural isomorphism  
$\arrow:\xhom[\K]{{{-}\tens{?}}}\d->\xhom[\K]{{{?}\tens{-}}}\d;$.

\begin{rem} \label{rem:nplds}
  By a \emph{\brdstaut\ category} we mean simply a \staut\ category
  together with a braiding for $\tens$.    
  By duality, this automatically induces a braiding for
  $\parr$---hence, there are actually two braidings:  
  we write $\tbraid$ for the braiding on $\tens$, 
  and $\pbraid$ for the braiding on $\parr$. 
   
  These two braidings cohere with one another in the sense that the four
  diagrams 
  \begin{diagram}@C=9pt{
      & (p \parr q) \tens r 
      \ar[dl]_-{\pbraid[p,q]\inverse\tens\id[r]}
      \ar[dr]^-{\tbraid[p \parr q,r]}
      &&& r \tens (q \parr p) 
      \ar[dl]_-{\tbraid[q \parr p,r]}
      \ar[dr]^-{\id[r]\tens\pbraid[p,q]\inverse}
      \\ (q \parr p) \tens r 
      \ar[ddd]_-{\rrcad[q,p,r]}
      \ar@{<.}[dr]|-{\tbraid[q \parr p,r]\inverse}
      && r \tens (p \parr q) 
      \ar[ddd]^-{\llcad[r,p,q]}
      \ar@{<.}[dl]|-{\id[r]\tens\pbraid[p,q]}
      & (q \parr p) \tens r 
      \ar[ddd]_-{\rrcad[q,p,r]}
      \ar@{<.}[dr]|-{\pbraid[p,q]\tens\id[r]}
      && r \tens (p \parr q) 
      \ar[ddd]^-{\llcad[r,p,q]}
      \ar@{<.}[dl]|-{\tbraid[p \parr q,r]\inverse}
      \\ & r \tens (q \parr p) 
      &&& (p \parr q) \tens r 
      \\ & q \parr (r \tens p)
      \ar@{<.}[dl]|-{\id[q]\parr\tbraid[r,p]} 
      \ar@{<.}[dr]|-{\pbraid[q,r \tens p]\inverse} 
      &&& (p \tens r) \parr q 
      \ar@{<.}[dl]|-{\pbraid[q,p \tens r]\inverse} 
      \ar@{<.}[dr]|-{\tbraid[r,p]\parr\id[q]} 
      \\ q \parr (p \tens r) 
      \ar[dr]_-{\pbraid[q,p \tens r]} 
      && (r \tens p) \parr q 
      \ar[dl]^-{\tbraid[r,p]\inverse\parr\id[q]} 
      & q \parr (p \tens r) 
      \ar[dr]_-{\id[q]\parr\tbraid[r,p]\inverse} 
      && (r \tens p) \parr q 
      \ar[dl]^-{\pbraid[q,r \tens p]} 
      \\ & (p \tens r) \parr q 
      &&& q \parr (r \tens p) }
  \end{diagram}
  hold (see Lemma \ref{appx:nplds}).   
  (Compare with the definition of \symstaut\ category
      in \cite[\S 3]{CocSee:ldc}.) 
  This means that the non-standard wire-crossings listed in 
  (the central column of) Figure \ref{fig:nplds} are well-defined. 
  It also entails that in the degenerate case 
  where $\tens=\parr$, $\e=\d$, $\rrcad=\assoc$ and
  $\llcad=\assoc\inverse$, 
  one can derive $\tbraid=\pbraid$; simply
  set $p=\e=\d$ in the first solid diagram above, and reduce
  appropriately.  
\end{rem}

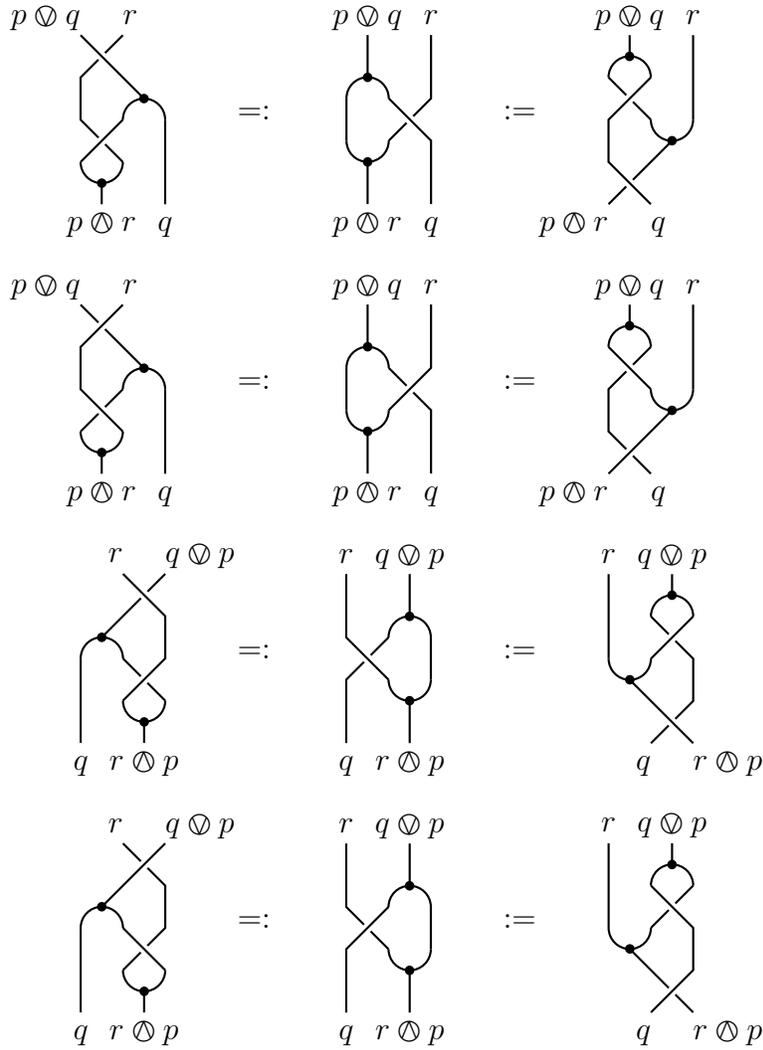
\begin{figure}
\begin{center}
  \begin{pspicture}[shift=-8](0,0)(8,16)
    \psarc(2,4){2}{180}{0}
    \psline(4,4)(0,8)(0,12)(4,16)
    \psdots[linecolor=white](2,6)(2,14)
    \psline(0,4)(4,8)
    \psarc(6,8){2}{0}{180}
    \psline(6,10)(4,12)(0,16)
    \psline(8,8)(8,0)
    \psline(2,2)(2,0)
    \psdots(2,2)(6,10)
    \uput[270](2,0){$\strut p \tens r$}
    \uput[270](8,0){$\strut q$}
    \uput[45](4,16){$\strut r$}
    \uput[135](0,16){$\strut p \parr q$}
  \end{pspicture}
  \qquad =: \qquad 
  \begin{pspicture}[shift=-8](0,-2)(8,14)
    \psarc(2,4){2}{180}{0}
    \psline(4,4)(8,8)(8,14)
    \psdots[linecolor=white](6,6)
    \psline(0,4)(0,8)
    \psarc(2,8){2}{0}{180}
    \psline(2,10)(2,14)
    \psline(4,8)(8,4)(8,-2)
    \psline(2,2)(2,-2)
    \psdots(2,2)(2,10)
    \uput[270](2,-2){$\strut p \tens r$}
    \uput[270](8,-2){$\strut q$}
    \uput[90](8,14){$\strut r$}
    \uput[90](2,14){$\strut p \parr q$}
  \end{pspicture}
  \qquad := \qquad 
  \begin{pspicture}[shift=-8](0,-4)(8,12)
    \psarc(6,4){2}{180}{0}
    \psline(8,4)(8,12)
    \psline(0,8)(4,4)
    \psline(6,2)(0,-4)
    \psdots[linecolor=white](2,6)(2,-2)
    \psarc(2,8){2}{0}{180}
    \psline(2,10)(2,12)
    \psline(4,8)(0,4)(0,0)(4,-4)
    \psdots(6,2)(2,10)
    \uput[225](0,-4){$\strut p \tens r$}
    \uput[315](4,-4){$\strut q$}
    \uput[90](8,12){$\strut r$}
    \uput[90](2,12){$\strut p \parr q$}
  \end{pspicture}

  \begin{pspicture}(0,0)(9,9)
  \end{pspicture} 

  \begin{pspicture}[shift=-8](0,0)(8,16)
    \psarc(2,4){2}{180}{0}
    \psline(0,4)(4,8)
    \psarc(6,8){2}{0}{180}
    \psline(6,10)(4,12)(0,16)
    \psline(8,8)(8,0)
    \psline(2,2)(2,0)
    \psdots(2,2)(6,10)
    \psdots[linecolor=white](2,6)(2,14)
    \psline(4,4)(0,8)(0,12)(4,16)
    \uput[270](2,0){$\strut p \tens r$}
    \uput[270](8,0){$\strut q$}
    \uput[45](4,16){$\strut r$}
    \uput[135](0,16){$\strut p \parr q$}
  \end{pspicture}
  \qquad =: \qquad 
  \begin{pspicture}[shift=-8](0,-2)(8,14)
    \psarc(2,4){2}{180}{0}
    \psline(0,4)(0,8)
    \psarc(2,8){2}{0}{180}
    \psline(2,10)(2,14)
    \psline(4,8)(8,4)(8,-2)
    \psline(2,2)(2,-2)
    \psdots(2,2)(2,10)
    \psdots[linecolor=white](6,6)
    \psline(4,4)(8,8)(8,14)
    \uput[270](2,-2){$\strut p \tens r$}
    \uput[270](8,-2){$\strut q$}
    \uput[90](8,14){$\strut r$}
    \uput[90](2,14){$\strut p \parr q$}
  \end{pspicture}
  \qquad := \qquad 
  \begin{pspicture}[shift=-8](0,-4)(8,12)
    \psarc(6,4){2}{180}{0}
    \psarc(2,8){2}{0}{180}
    \psline(8,4)(8,12)
    \psline(2,10)(2,12)
    \psline(4,8)(0,4)(0,0)(4,-4)
    \psdots(6,2)(2,10)
    \psdots[linecolor=white](2,6)(2,-2)
    \psline(0,8)(4,4)
    \psline(6,2)(0,-4)
    \uput[225](0,-4){$\strut p \tens r$}
    \uput[315](4,-4){$\strut q$}
    \uput[90](8,12){$\strut r$}
    \uput[90](2,12){$\strut p \parr q$}
  \end{pspicture}

  \begin{pspicture}(0,0)(-9,9)
  \end{pspicture} 

  \begin{pspicture}[shift=-8](-8,0)(0,16)
    \psarc(-2,4){2}{180}{0}
    \psline(0,4)(-4,8)
    \psarc(-6,8){2}{0}{180}
    \psline(-6,10)(-4,12)(0,16)
    \psline(-8,8)(-8,0)
    \psline(-2,2)(-2,0)
    \psdots(-2,2)(-6,10)
    \psdots[linecolor=white](-2,6)(-2,14)
    \psline(-4,4)(0,8)(0,12)(-4,16)
    \uput[270](-2,0){$\strut r \tens p$}
    \uput[270](-8,0){$\strut q$}
    \uput[135](-4,16){$\strut r$}
    \uput[45](0,16){$\strut q \parr p$}
  \end{pspicture}
  \qquad =: \qquad 
  \begin{pspicture}[shift=-8](-8,-2)(0,14)
    \psarc(-2,4){2}{180}{0}
    \psline(0,4)(0,8)
    \psarc(-2,8){2}{0}{180}
    \psline(-2,10)(-2,14)
    \psline(-4,8)(-8,4)(-8,-2)
    \psline(-2,2)(-2,-2)
    \psdots(-2,2)(-2,10)
    \psdots[linecolor=white](-6,6)
    \psline(-4,4)(-8,8)(-8,14)
    \uput[270](-2,-2){$\strut r \tens p$}
    \uput[270](-8,-2){$\strut q$}
    \uput[90](-8,14){$\strut r$}
    \uput[90](-2,14){$\strut q \parr p$}
  \end{pspicture}
  \qquad := \qquad 
  \begin{pspicture}[shift=-8](-8,-4)(0,12)
    \psarc(-6,4){2}{180}{0}
    \psarc(-2,8){2}{0}{180}
    \psline(-8,4)(-8,12)
    \psline(-2,10)(-2,12)
    \psline(-4,8)(0,4)(0,0)(-4,-4)
    \psdots(-6,2)(-2,10)
    \psdots[linecolor=white](-2,6)(-2,-2)
    \psline(0,8)(-4,4)
    \psline(-6,2)(0,-4)
    \uput[315](0,-4){$\strut r \tens p$}
    \uput[225](-4,-4){$\strut q$}
    \uput[90](-8,12){$\strut r$}
    \uput[90](-2,12){$\strut q \parr p$}
  \end{pspicture}

  \begin{pspicture}(0,0)(9,9)
  \end{pspicture} 

  \begin{pspicture}[shift=-8](-8,0)(0,16)
    \psarc(-2,4){2}{180}{0}
    \psline(-4,4)(0,8)(0,12)(-4,16)
    \psdots[linecolor=white](-2,6)(-2,14)
    \psline(0,4)(-4,8)
    \psarc(-6,8){2}{0}{180}
    \psline(-6,10)(-4,12)(0,16)
    \psline(-8,8)(-8,0)
    \psline(-2,2)(-2,0)
    \psdots(-2,2)(-6,10)
    \uput[270](-2,0){$\strut r \tens p$}
    \uput[270](-8,0){$\strut q$}
    \uput[135](-4,16){$\strut r$}
    \uput[45](0,16){$\strut q \parr p$}
  \end{pspicture}
  \qquad =: \qquad 
  \begin{pspicture}[shift=-8](-8,-2)(0,14)
    \psarc(-2,4){2}{180}{0}
    \psline(-4,4)(-8,8)(-8,14)
    \psdots[linecolor=white](-6,6)
    \psline(0,4)(0,8)
    \psarc(-2,8){2}{0}{180}
    \psline(-2,10)(-2,14)
    \psline(-4,8)(-8,4)(-8,-2)
    \psline(-2,2)(-2,-2)
    \psdots(-2,2)(-2,10)
    \uput[270](-2,-2){$\strut r \tens p$}
    \uput[270](-8,-2){$\strut q$}
    \uput[90](-8,14){$\strut r$}
    \uput[90](-2,14){$\strut q \parr p$}
  \end{pspicture}
  \qquad := \qquad 
  \begin{pspicture}[shift=-8](-8,-4)(0,12)
    \psarc(-6,4){2}{180}{0}
    \psline(-8,4)(-8,12)
    \psline(0,8)(-4,4)
    \psline(-6,2)(0,-4)
    \psdots[linecolor=white](-2,6)(-2,-2)
    \psarc(-2,8){2}{0}{180}
    \psline(-2,10)(-2,12)
    \psline(-4,8)(0,4)(0,0)(-4,-4)
    \psdots(-6,2)(-2,10)
    \uput[315](0,-4){$\strut r \tens p$}
    \uput[225](-4,-4){$\strut q$}
    \uput[90](-8,12){$\strut r$}
    \uput[90](-2,12){$\strut q \parr p$}
  \end{pspicture}
\end{center}
  \caption{Non-planar linear distributions in a braided \staut\ category} 
  \label{fig:nplds}
\end{figure}

In light of the remarks above, 
it will be convenient to use the following (provisional) terminology.  

\begin{defs} \label{def:balance}
Let 
$\balance$ be a natural isomorphism 
$\arrow:\Id[\K]->\Id[\K];$. 
Then we call $\balance$:
a \emph{{\ts}balance for $\K$} if it is a balance for $\tbraid$---that
is, if \tsbax\ holds;
a \emph{{\ps}balance for $\K$} if it is a balance for $\pbraid$---that
is, if \psbax\ holds; 
a \emph{balance for $\K$} if it is both a {\ts}balance and a 
{\ps}balance---that is, if both \tsbax\ and \psbax\ hold.
\begin{diagram}{ 
  p \tens q \ar[r]^-{\tbraid[p,q]} \ar[d]_-{\balance[p \tens q]} 
  \ar@{}[dr]|-{\tsbax}
  & q \tens p \ar[d]^-{\balance[q]\tens\balance[p]}
  & p \parr q \ar[r]^-{\tbraid[p,q]} \ar[d]_-{\balance[p \parr q]} 
  \ar@{}[dr]|-{\psbax}
  & q \parr p \ar[d]^-{\balance[q]\parr\balance[p]}
  \\ p \tens q 
  & q \tens p \ar[l]^-{\tbraid[q,p]} 
  & p \parr q 
  & q \parr p \ar[l]^-{\tbraid[q,p]} }
\end{diagram}
\end{defs}
Note that \tsbax\ entails $\balance[\e]=\id[\e]$ and that \psbax\
entails $\balance[\d]=\id[\d]$ for the same reason as
$\tbin\entails\tnul$ and $\pbin\entails\pnul$---see, again,
Lemma \ref{appx:snis}.  

\begin{thm} \label{thm:balance}
There are bijective correspondences between: 
  {\ts\cthing}s and {\ts}balances; 
  {\ps\cthing}s and {\ps}balances; 
  {\cthing}s and balances.
\end{thm}

This theorem will be proven by series of lemmata following the
graphical intuitions laid out in section \ref{sec:intr}.  
(But, since our ambient \staut\ category is only assumed to be
braided, we use wire/strings instead of tape/ribbons.)  

\begin{lem}
  Let $\nubalance$ be the \nt\ whose components are given by 
  (the common composite of) the diagram below---or, equivalently, by
  (the common value of) the string diagrams which follow. 
  \begin{diagram}@C=3.75pc{
      p \ar[r]^-{\rulaw[p]\inverse}
      \ar@/_/[dr]_-{\lulaw[p]\inverse}
      & p \tens \e \ar[r]^-{\id[p]\tens\tau}
      & p \tens (\perp{p} \parr p) 
      \ar[r]^-{\llcad[p,\perp{p},p](\cycle[p])}
      & (p \tens \prep{p}) \parr p 
      \ar[d]|-{\tbraid[p,\prep{p}]\inverse\parr\id[p]} 
      \\ 
      & \e \tens p \ar[r]^-{\tau\tens\id[p]} \ar[u]|-{\tbraid[\e,p]}
      & (\perp{p} \parr p) \tens p 
      \ar[d]|-{\pbraid[\perp{p},p]\inverse\tens\id[p]}
      \ar[u]|-{\tbraid[\perp{p} \parr p,p]}
      \ar@{.>}[r]
      & (\prep{p} \tens p) \parr p 
      \ar[r]^-{\gamma\parr\id[p]}
      & \d \parr p \ar@/^/[dr]^-{\lulaw[p]} 
      \ar[d]|-{\pbraid[\d,p]}
      \\ 
      & & (p \parr \perp{p}) \tens p 
      \ar[r]_-{\rrcad[p,\perp{p},p](\cycle[p])}
      & p \parr (\prep{p} \tens p) 
      \ar[u]|-{\pbraid[p,\prep{p} \tens p]} 
      \ar[r]_-{\id[p]\parr\gamma}
      & p \parr \d \ar[r]_-{\rulaw[p]} 
      & p }
  \end{diagram}
  \begin{center}
    \begin{pspicture}[shift=-8](0,0)(8,16)
      \psarc(2,4){2}{180}{0}
      \psline(4,4)(0,8)(0,14) 
      \psdots[linecolor=white](2,6) 
      \psline(0,4)(4,8)(4,12)
      \psarc(6,12){2}{0}{180}
      \psline(8,12)(8,2)
      \psframe[fillstyle=solid,fillcolor=boxcol](2.5,8.5)(5.5,11.5)
      \rput(4,10){\tiny $\strut \cycle[p]$}
      \uput[270](8,2){$\strut p$}
      \uput[90](0,14){$\strut p$}
    \end{pspicture}
    \qquad = \qquad 
    \begin{pspicture}[shift=-8](0,-2)(8,14)
      \psarc(2,4){2}{180}{0}
      \psline(4,4)(8,8)(8,12)
      \psdots[linecolor=white](6,6)
      \psline(0,4)(0,8)
      \psarc(2,8){2}{0}{180}
      \psline(4,8)(8,4)(8,0)
      \psframe[fillstyle=solid,fillcolor=boxcol](-1.5,4.5)(1.5,7.5)
      \rput(0,6){\tiny $\strut \cycle[p]$}
      \uput[270](8,0){$\strut p$}
      \uput[90](8,12){$\strut p$}
    \end{pspicture}
    \qquad = \qquad 
    \begin{pspicture}[shift=-9](0,-5)(8,13)
      \psarc(6,0){2}{180}{0}
      \psline(8,0)(8,10)
      \psline(0,8)(4,4)(4,0)
      \psdots[linecolor=white](2,6) 
      \psarc(2,8){2}{0}{180}
      \psline(4,8)(0,4)(0,-2) 
      \psframe[fillstyle=solid,fillcolor=boxcol](2.5,0.5)(5.5,3.5)
      \rput(4,2){\tiny $\strut \cycle[p]$}
      \uput[270](0,-2){$\strut p$}
      \uput[90](8,10){$\strut p$}
    \end{pspicture}
  \end{center}
  If $\cycle$ is a \ts\cthing, then $\nubalance$ is a {\ts}balance.
  Similarly, 
  if $\cycle$ is a \ps\cthing, then $\nubalance$ is a {\ps}balance.
  Hence, if $\cycle$ is a \cthing, then $\nubalance$ is a balance.
\end{lem}
\begin{proof}
  We give a graphical proof of $\tbin\implies\tsbax$:
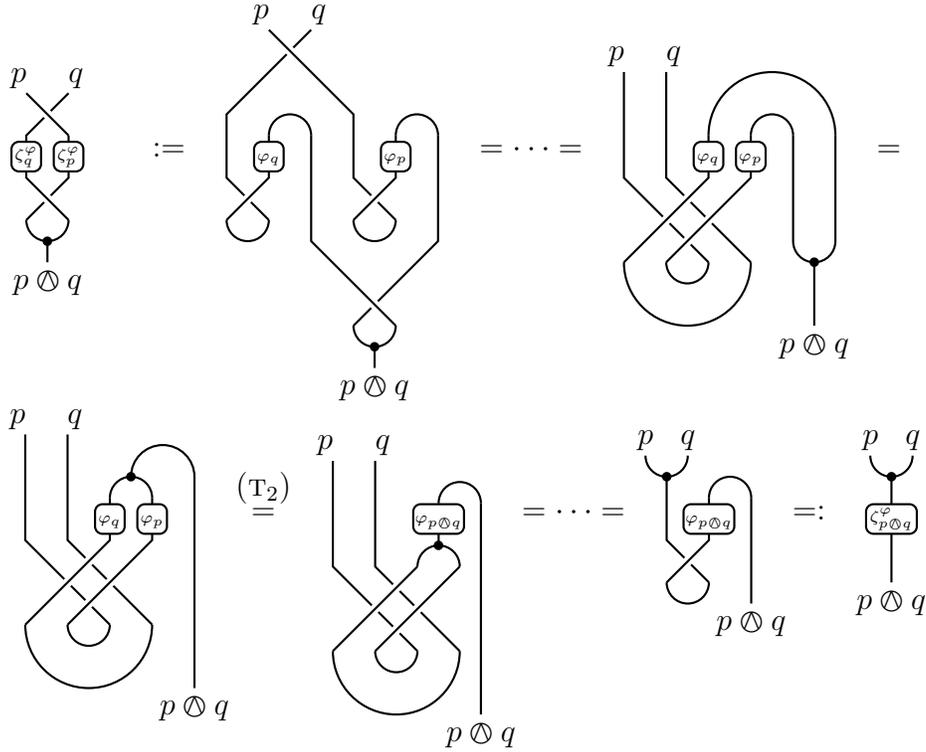
  \begin{center}
    \begin{pspicture}[shift=-10](0,0)(8,16)
      \psarc(2,4){2}{180}{0}
      \psline(0,4)(4,8)(4,12)
      \psdots[linecolor=white](2,6) 
      \psline(4,4)(0,8)(0,12)(4,16) 
      \psdots[linecolor=white](2,14)
      \psline(4,12)(0,16) 
      \psline(2,2)(2,0)
      \psdots(2,2)
      \psframe[fillstyle=solid,fillcolor=boxcol](-1.5,8.5)(1.5,11.5)
      \rput(0,10){\tiny $\strut \nubalance[q]$}
      \psframe[fillstyle=solid,fillcolor=boxcol](2.5,8.5)(5.5,11.5)
      \rput(4,10){\tiny $\strut \nubalance[p]$}
      \uput[270](2,0){$\strut p \tens q$}
      \uput[45](4,16){$\strut q$}
      \uput[135](0,16){$\strut p$}
    \end{pspicture}
    \quad := \quad 
    \begin{pspicture}[shift=-23](0,-13)(20,25) 
      \psarc(2,4){2}{180}{0}
      \psline(4,4)(0,8)(0,14)(8,22) 
      \psdots[linecolor=white](2,6)(6,20) 
      \psline(0,4)(4,8)(4,12)
      \psarc(6,12){2}{0}{180}
      \psframe[fillstyle=solid,fillcolor=boxcol](2.5,8.5)(5.5,11.5)
      \rput(4,10){\tiny $\strut \cycle[q]$}
      \psarc(14,4){2}{180}{0}
      \psline(16,4)(12,8)(12,14)(4,22) 
      \psdots[linecolor=white](14,6) 
      \psline(12,4)(16,8)(16,12)
      \psarc(18,12){2}{0}{180}
      \psline(20,12)(20,2)(12,-6)
      \psdots[linecolor=white](14,-4) 
      \psline(8,12)(8,2)(16,-6)
      \psframe[fillstyle=solid,fillcolor=boxcol](14.5,8.5)(17.5,11.5)
      \rput(16,10){\tiny $\strut \cycle[p]$}
      \psarc(14,-6){2}{180}{0}
      \psline(14,-8)(14,-10)
      \psdots(14,-8)
      \uput[45](8,22){$\strut q$}
      \uput[135](4,22){$\strut p$}
      \uput[270](14,-10){$\strut p \tens q$}
    \end{pspicture}
    \quad $= \cdots =$ \quad 
    \begin{pspicture}[shift=-16](4,-6)(24,18)
      \psarc(10,0){6}{180}{0}
      \psline(16,0)(8,8)(8,18) 
      \psdots[linecolor=white](2,6)(6,20) 
      \psarc(18,12){6}{0}{180}
      \psarc(10,0){2}{180}{0}
      \psline(12,0)(4,8)(4,18) 
      \psdots[linecolor=white](10,2)(8,4)(12,4)(10,6) 
      \psline(4,0)(12,8)(12,12)
      \psline(8,0)(16,8)(16,12)
      \psarc(18,12){2}{0}{180}
      \psline(20,12)(20,2) 
      \psline(24,12)(24,2) 
      \psarc(22,2){2}{180}{0}
      \psline(22,0)(22,-6)
      \psdots(22,0)
      \psframe[fillstyle=solid,fillcolor=boxcol](14.5,8.5)(17.5,11.5)
      \rput(16,10){\tiny $\strut \cycle[p]$}
      \psframe[fillstyle=solid,fillcolor=boxcol](10.5,8.5)(13.5,11.5)
      \rput(12,10){\tiny $\strut \cycle[q]$}
      \uput[45](8,18){$\strut q$}
      \uput[135](4,18){$\strut p$}
      \uput[270](22,-6){$\strut p \tens q$}
    \end{pspicture}
    \quad = \quad 


    \begin{pspicture}[shift=-19](4,-9)(20,21) 
      \psarc(10,0){6}{180}{0}
      \psline(16,0)(8,8)(8,18) 
      \psdots[linecolor=white](2,6)(6,20) 
      \psarc(10,0){2}{180}{0}
      \psline(12,0)(4,8)(4,18) 
      \psdots[linecolor=white](10,2)(8,4)(12,4)(10,6) 
      \psline(4,0)(12,8)(12,12)
      \psline(8,0)(16,8)(16,12)
      \psarc(14,12){2}{0}{180}
      \psline(20,14)(20,-6)
      \psdots(14,14)
      \psarc(17,14){3}{0}{180}
      \psframe[fillstyle=solid,fillcolor=boxcol](14.5,8.5)(17.5,11.5)
      \rput(16,10){\tiny $\strut \cycle[p]$}
      \psframe[fillstyle=solid,fillcolor=boxcol](10.5,8.5)(13.5,11.5)
      \rput(12,10){\tiny $\strut \cycle[q]$}
      \uput[45](8,18){$\strut q$}
      \uput[135](4,18){$\strut p$}
      \uput[270](20,-6){$\strut p \tens q$}
    \end{pspicture}
    \quad $\stackrel{\displaystyle\tbin}=$ \quad 
    \begin{pspicture}[shift=-21.5](4,-9)(18,21) 
      \psarc(10,0){6}{180}{0}
      \psline(16,0)(8,8)(8,18) 
      \psdots[linecolor=white](2,6)(6,20) 
      \psarc(10,0){2}{180}{0}
      \psline(12,0)(4,8)(4,18) 
      \psdots[linecolor=white](10,2)(8,4)(12,4)(10,6) 
      \psline(4,0)(12,8) 
      \psline(8,0)(16,8) 
      \psarc(14,8){2}{0}{180}
      \psline(18,14)(18,-6)
      \psdots(14,10)
      \psarc(16,14){2}{0}{180}
      \psline(14,14)(14,10)
      \psframe[fillstyle=solid,fillcolor=boxcol](11.5,11)(16.5,14)
      \rput(14,12.5){\tiny $\strut \cycle[p \tens q]$}
      \uput[45](8,18){$\strut q$}
      \uput[135](4,18){$\strut p$}
      \uput[270](18,-6){$\strut p \tens q$}
    \end{pspicture}
    \quad $= \cdots =$ \quad 
    \begin{pspicture}[shift=-10](0,0)(8,16)
      \psarc(2,4){2}{180}{0}
      \psline(4,4)(0,8)(0,14) 
      \psdots[linecolor=white](2,6) 
      \psline(0,4)(4,8)(4,12)
      \psarc(6,12){2}{0}{180}
      \psline(8,12)(8,2)
      \psframe[fillstyle=solid,fillcolor=boxcol](1.5,8.5)(6.5,11.5)
      \rput(4,10){\tiny $\strut \cycle[p \tens q]$}
      \uput[270](8,2){$\strut p \tens q$}
      \psarc(0,16){2}{180}{0}
      \psdots(0,14)
      \uput[90](-2,16){$\strut p$}
      \uput[90](2,16){$\strut q$}
    \end{pspicture}
    \quad =: \quad 
    \begin{pspicture}[shift=-10](-2,0)(2,16)
      \psline(0,4)(0,14) 
      \psframe[fillstyle=solid,fillcolor=boxcol](-2.5,8.5)(2.5,11.5)
      \rput(0,10){\tiny $\strut \nubalance[p \tens q]$}
      \uput[270](0,4){$\strut p \tens q$}
      \psarc(0,16){2}{180}{0}
      \psdots(0,14)
      \uput[90](-2,16){$\strut p$}
      \uput[90](2,16){$\strut q$}
    \end{pspicture}
  \end{center}
 
The proof of $\pbin\implies\psbax$ 
is exactly dual. 
\end{proof}

\begin{lem}
Let $\nuCycle$ be the operation on external hom-sets given below. 
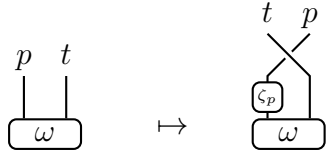
\[ \nuCycle[p,t](\omega) :=
\braid[p,t]\fo(\balance[p]\tens\id[t])\fo\omega \]
\begin{center}
  \begin{pspicture}[shift=-1.5](-1.5,-3)(5.5,7)
    \psline(0,0)(0,4)
    \psline(4,0)(4,4)
    \psframe[fillstyle=solid,fillcolor=boxcol](-1.5,0)(5.5,-3)
    \rput(2,-1.5){$\omega$}
    \uput[90](0,4){$\strut p$}
    \uput[90](4,4){$\strut t$}
  \end{pspicture}
  \qquad{$\mapsto$}\qquad
  \begin{pspicture}[shift=-1.5](-1.5,-3)(5.5,11)
    \psline(0,0)(0,4)(4,8)
    \psdots[linecolor=white](2,6)
    \psline(4,0)(4,4)(0,8)
    \psframe[fillstyle=solid,fillcolor=boxcol](-1.5,0.5)(1.5,3.5)
    \rput(0,2){\tiny $\balance[p]$}
    \psframe[fillstyle=solid,fillcolor=boxcol](-1.5,0)(5.5,-3)
    \rput(2,-1.5){$\omega$}
    \uput[90](4,8){$\strut p$}
    \uput[90](0,8){$\strut t$}
  \end{pspicture}
\end{center}
If $\balance$ is a {\ts}balance, then $\nuCycle$ is a \ts\cthing.
Similarly, if $\balance$ is a {\ps}balance, then $\nuCycle$ is
a \ps\cthing. 
Hence, if $\balance$ is a balance, then $\nuCycle$ is a \cthing. 
\end{lem}
\begin{proof}
  To prove $\tsbax\implies\etwo$, we need to show that the following
  equality holds, for every arrow  
  $\arrow\omega:(p \tens q)\tens t->\d;$.
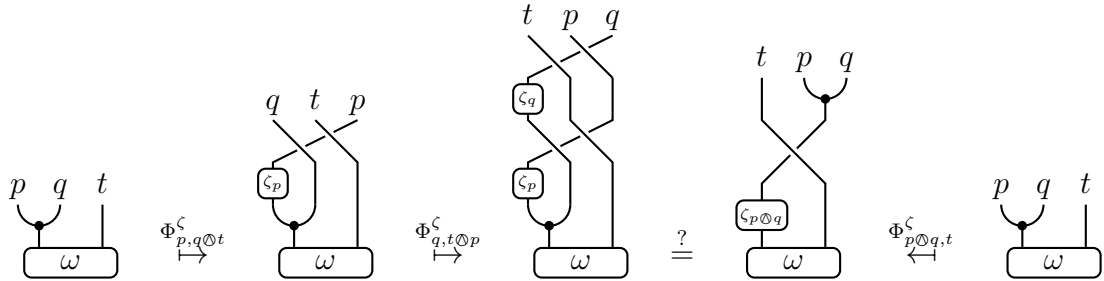
  \begin{center}
    \begin{pspicture}[shift=-1.5](-1.5,-3)(9.5,7)
      \psarc(2,4){2}{180}{0}
      \psline(2,0)(2,2)
      \psdots(2,2)
      \psline(8,0)(8,4)
      \psframe[fillstyle=solid,fillcolor=boxcol](0.5,0)(9.5,-3)
      \rput(5,-1.5){$\omega$}
      \uput[90](0,4){$\strut p$}
      \uput[90](4,4){$\strut q$}
      \uput[90](8,4){$\strut t$}
    \end{pspicture}
    \quad{$\stackrel{\nuCycle[p,q \tens t]}\mapsto$}\quad
    \begin{pspicture}[shift=-1.5](-1.5,-3)(9.5,15)
      \psarc(2,4){2}{180}{0}
      \psline(2,0)(2,2)
      \psdots(2,2)
      \psline(0,4)(0,8)(8,12)
      \psdots[linecolor=white](2.667,9.333)(5.333,10.667)
      \psline(4,4)(4,8)(0,12)
      \psline(8,0)(8,8)(4,12) 
      \psframe[fillstyle=solid,fillcolor=boxcol](-1.5,4.5)(1.5,7.5)
      \rput(0,6){\tiny $\balance[p]$}
      \psframe[fillstyle=solid,fillcolor=boxcol](0.5,0)(9.5,-3)
      \rput(5,-1.5){$\omega$}
      \uput[90](8,12){$\strut p$}
      \uput[90](4,12){$\strut t$}
      \uput[90](0,12){$\strut q$}
    \end{pspicture}
    \quad{$\stackrel{\nuCycle[q,t \tens p]}\mapsto$}\quad
    \begin{pspicture}[shift=-1.5](-1.5,-3)(9.5,23)
      \psarc(2,4){2}{180}{0}
      \psline(2,0)(2,2)
      \psdots(2,2)
      \psline(0,4)(0,8)(8,12)(8,16)
      \psdots[linecolor=white](2.667,9.333)(5.333,10.667)
      \psline(4,4)(4,8)(0,12)(0,16)(8,20)
      \psdots[linecolor=white](2.667,17.333)(5.333,18.667)
      \psline(8,0)(8,8)(4,12)(4,16)(0,20) 
      \psline(8,16)(4,20)
      \psframe[fillstyle=solid,fillcolor=boxcol](-1.5,4.5)(1.5,7.5)
      \rput(0,6){\tiny $\balance[p]$}
      \psframe[fillstyle=solid,fillcolor=boxcol](-1.5,12.5)(1.5,15.5)
      \rput(0,14){\tiny $\balance[q]$}
      \psframe[fillstyle=solid,fillcolor=boxcol](0.5,0)(9.5,-3)
      \rput(5,-1.5){$\omega$}
      \uput[90](0,20){$\strut t$}
      \uput[90](4,20){$\strut p$}
      \uput[90](8,20){$\strut q$}
    \end{pspicture}
    \quad $\stackrel?=$ \quad
    \begin{pspicture}[shift=-1.5](-0.5,-3)(10,16)
      \psline(2,0)(2,6)(8,12)(8,14)
      \psarc(8,16){2}{180}{0}
      \psdots(8,14)
      \psframe[fillstyle=solid,fillcolor=boxcol](-0.5,1.5)(4.5,4.5)
      \rput(2,3){\tiny $\strut \balance[p \tens q]$}
      \psdots[linecolor=white](5,9)
      \psline(8,0)(8,6)(2,12)(2,16) 
      \psframe[fillstyle=solid,fillcolor=boxcol](0.5,0)(9.5,-3)
      \rput(5,-1.5){$\omega$}
      \uput[90](2,16){$\strut t$}
      \uput[90](6,16){$\strut p$}
      \uput[90](10,16){$\strut q$}
    \end{pspicture}
    \quad{$\stackrel{\nuCycle[p \tens q,t]}\mapsfrom$}\quad
    \begin{pspicture}[shift=-1.5](-1.5,-3)(9.5,7)
      \psarc(2,4){2}{180}{0}
      \psline(2,0)(2,2)
      \psdots(2,2)
      \psline(8,0)(8,4)
      \psframe[fillstyle=solid,fillcolor=boxcol](0.5,0)(9.5,-3)
      \rput(5,-1.5){$\omega$}
      \uput[90](0,4){$\strut p$}
      \uput[90](4,4){$\strut q$}
      \uput[90](8,4){$\strut t$}
    \end{pspicture}
  \end{center}
  This is easily established by the argument below. 
  
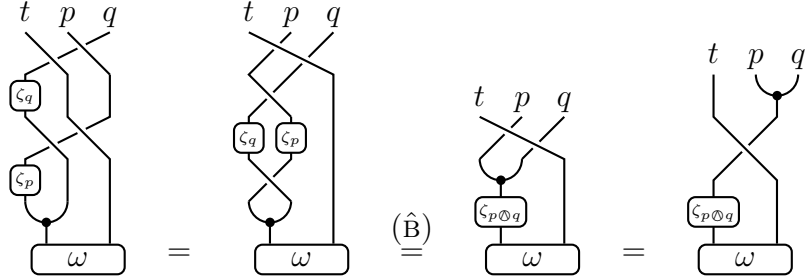
\begin{center}
    \begin{pspicture}[shift=-1.5](-1.5,-3)(9.5,23)
      \psarc(2,4){2}{180}{0}
      \psline(2,0)(2,2)
      \psdots(2,2)
      \psline(0,4)(0,8)(8,12)(8,16)
      \psdots[linecolor=white](2.667,9.333)(5.333,10.667)
      \psline(4,4)(4,8)(0,12)(0,16)(8,20)
      \psdots[linecolor=white](2.667,17.333)(5.333,18.667)
      \psline(8,0)(8,8)(4,12)(4,16)(0,20) 
      \psline(8,16)(4,20)
      \psframe[fillstyle=solid,fillcolor=boxcol](-1.5,4.5)(1.5,7.5)
      \rput(0,6){\tiny $\balance[p]$}
      \psframe[fillstyle=solid,fillcolor=boxcol](-1.5,12.5)(1.5,15.5)
      \rput(0,14){\tiny $\balance[q]$}
      \psframe[fillstyle=solid,fillcolor=boxcol](0.5,0)(9.5,-3)
      \rput(5,-1.5){$\omega$}
      \uput[90](0,20){$\strut t$}
      \uput[90](4,20){$\strut p$}
      \uput[90](8,20){$\strut q$}
    \end{pspicture}
    \quad = \quad
    \begin{pspicture}[shift=-1.5](-1.5,-3)(9.5,23)
      \psarc(2,4){2}{180}{0}
      \psline(0,4)(4,8)(4,12)
      \psdots[linecolor=white](2,6) 
      \psline(4,4)(0,8)(0,12)(4,16)(8,20) 
      \psdots[linecolor=white](2,14)
      \psline(4,12)(0,16)(4,20)  
      \psline(2,2)(2,0)
      \psdots(2,2)
      \psframe[fillstyle=solid,fillcolor=boxcol](-1.5,8.5)(1.5,11.5)
      \rput(0,10){\tiny $\strut \balance[q]$}
      \psframe[fillstyle=solid,fillcolor=boxcol](2.5,8.5)(5.5,11.5)
      \rput(4,10){\tiny $\strut \balance[p]$}
      \psdots[linecolor=white](2.667,18.667)(5.333,17.333)
      \psline(8,0)(8,16)(0,20) 
      \psframe[fillstyle=solid,fillcolor=boxcol](0.5,0)(9.5,-3)
      \rput(5,-1.5){$\omega$}
      \uput[90](0,20){$\strut t$}
      \uput[90](4,20){$\strut p$}
      \uput[90](8,20){$\strut q$}
    \end{pspicture}
    \quad $\stackrel{\displaystyle\tsbax}{=}$ \quad
    \begin{pspicture}[shift=-1.5](-0.5,-3)(9.5,12)
      \psarc(2,8){2}{180}{0}
      \psline(4,8)(8,12)
      \psline(0,8)(4,12) 
      \psline(2,0)(2,6)
      \psdots(2,6)
      \psframe[fillstyle=solid,fillcolor=boxcol](-0.5,1.5)(4.5,4.5)
      \rput(2,3){\tiny $\strut \balance[p \tens q]$}
      \psdots[linecolor=white](2.667,10.667)(5.333,9.333)
      \psline(8,0)(8,8)(0,12) 
      \psframe[fillstyle=solid,fillcolor=boxcol](0.5,0)(9.5,-3)
      \rput(5,-1.5){$\omega$}
      \uput[90](0,12){$\strut t$}
      \uput[90](4,12){$\strut p$}
      \uput[90](8,12){$\strut q$}
    \end{pspicture}
    \quad = \quad
    \begin{pspicture}[shift=-1.5](-0.5,-3)(10,16)
      \psline(2,0)(2,6)(8,12)(8,14)
      \psarc(8,16){2}{180}{0}
      \psdots(8,14)
      \psframe[fillstyle=solid,fillcolor=boxcol](-0.5,1.5)(4.5,4.5)
      \rput(2,3){\tiny $\strut \balance[p \tens q]$}
      \psdots[linecolor=white](5,9)
      \psline(8,0)(8,6)(2,12)(2,16) 
      \psframe[fillstyle=solid,fillcolor=boxcol](0.5,0)(9.5,-3)
      \rput(5,-1.5){$\omega$}
      \uput[90](2,16){$\strut t$}
      \uput[90](6,16){$\strut p$}
      \uput[90](10,16){$\strut q$}
    \end{pspicture}
  \end{center}

To prove $\psbax\implies\mtwo$,
we need to show that, for every pair of arrows  
$\arrow\omega:p \tens t->\d;$ and $\arrow\psi:q \tens s->\d;$,
the result of chasing $(\omega,\psi)$ along the lower path of \mtwo\
equals the result of chasing it along the upper path. 
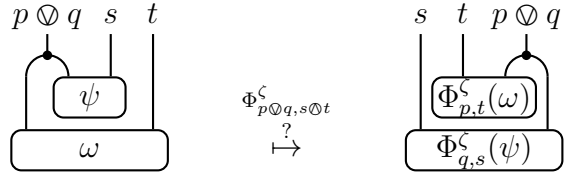
\begin{center}
  \begin{pspicture}[shift=-1.5](-5.5,-9)(9.5,7)
    \psline(4,0)(4,4)
    \psframe[fillstyle=solid,fillcolor=boxcol](-1.5,0)(5.5,-4)
    \rput(2,-2){$\psi$}
    \uput[90](4,4){$\strut s$}
    \psline(-4,-5)(-4,0)
    \psline(8,-5)(8,4)
    \psframe[fillstyle=solid,fillcolor=boxcol](-5.5,-5)(9.5,-9)
    \rput(2,-7){$\omega$}
    \uput[90](8,4){$\strut t$}
    \psarc(-2,0){2}{0}{180}
    \psdots(-2,2)
    \psline(-2,2)(-2,4)
    \uput[90](-2,4){$\strut p \parr q$}
  \end{pspicture}
  {\qquad$\stackrel{\nuCycle[p \parr q,s \tens t]}{\stackrel{?}{\mapsto}}$\qquad}
  \begin{pspicture}[shift=-1.5](-5.5,-9)(9.5,7)
    \psline(0,0)(0,4)
    \psframe[fillstyle=solid,fillcolor=boxcol](-3,0)(7,-4)
    \rput(2,-2){$\nuCycle[p,t](\omega)$}
    \uput[90](0,4){$\strut t$}
    \psline(-4,-5)(-4,4)
    \psline(8,-5)(8,0)
    \psframe[fillstyle=solid,fillcolor=boxcol](-5.5,-5)(9.5,-9)
    \rput(2,-7){$\nuCycle[q,s](\psi)$}
    \uput[90](-4,4){$\strut s$}
    \psarc(6,0){2}{0}{180}
    \psdots(6,2)
    \psline(6,2)(6,4)
    \uput[90](6,4){$\strut p \parr q$}
  \end{pspicture}
\end{center}
Again, this is easily proven, as follows.
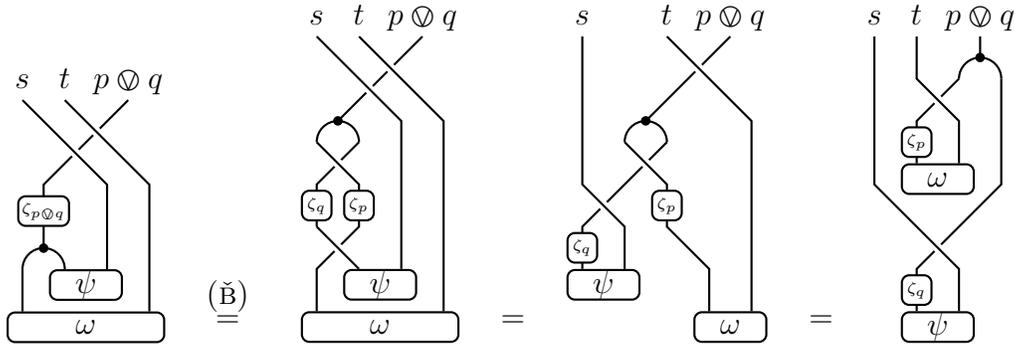
\begin{center}
  \begin{pspicture}[shift=-1.5](-5.5,-7)(9.5,19)
    \psline(-2,2)(-2,8)(6,16)
    \psarc(-2,0){2}{0}{180}
    \psline(-4,-4)(-4,0)
    \psdots(-2,2)
    \psdots[linecolor=white](1,11)(3,13)
    \psline(4,0)(4,8)(-4,16)
    \psline(8,-4)(8,8)(0,16)
    \psframe[fillstyle=solid,fillcolor=boxcol](-4.5,4)(0.5,7)
    \rput(-2,5.5){\tiny $\balance[p \parr q]$}
    \psframe[fillstyle=solid,fillcolor=boxcol](-1.5,0)(5.5,-3)
    \rput(2,-1.5){$\psi$}
    \psframe[fillstyle=solid,fillcolor=boxcol](-5.5,-4)(9.5,-7)
    \rput(2,-5.5){$\omega$}
    \uput[90](-4,16){$\strut s$}
    \uput[90](0,16){$\strut t$}
    \uput[90](6,16){$\strut p \parr q$}
  \end{pspicture}
  {\quad $\stackrel{\displaystyle\psbax}{=}$ \quad}
  \begin{pspicture}[shift=-1.5](-5.5,-13)(9.5,19)
    \psline(-2,8)(6,16)
    \psarc(-2,6){2}{0}{180}
    \psline(0,6)(-4,2)(-4,-2)
    \psdots[linecolor=white](-2,4)
    \psline(-4,6)(0,2)(0,-2)(-4,-6)(-4,-10)
    \psdots(-2,8)
    \psdots[linecolor=white](1,11)(3,13)(-2,-4)
    \psline(-4,-2)(0,-6)
    \psline(4,-6)(4,8)(-4,16)
    \psline(8,-10)(8,8)(0,16)
    \psframe[fillstyle=solid,fillcolor=boxcol](-5.5,-1.5)(-2.5,1.5)
    \rput(-4,0){\tiny $\balance[q]$}
    \psframe[fillstyle=solid,fillcolor=boxcol](-1.5,-1.5)(1.5,1.5)
    \rput(0,0){\tiny $\balance[p]$}
    \psframe[fillstyle=solid,fillcolor=boxcol](-1.5,-6)(5.5,-9)
    \rput(2,-7.5){$\psi$}
    \psframe[fillstyle=solid,fillcolor=boxcol](-5.5,-10)(9.5,-13)
    \rput(2,-11.5){$\omega$}
    \uput[90](-4,16){$\strut s$}
    \uput[90](0,16){$\strut t$}
    \uput[90](6,16){$\strut p \parr q$}
  \end{pspicture}
  {\quad=\quad}
  \begin{pspicture}[shift=-1.5](-9.5,-13)(9.5,19)
    \psline(-2,8)(6,16)
    \psarc(-2,6){2}{0}{180}
    \psline(0,6)(-8,-2)(-8,-6)
    \psdots[linecolor=white](-2,4)(-6,0)
    \psline(-4,6)(0,2)(0,-2)(4,-6)(4,-10)
    \psdots(-2,8)
    \psdots[linecolor=white](3,13)
    \psline(-4,-6)(-4,-2)(-8,2)(-8,16) 
    \psline(8,-10)(8,8)(0,16)
    \psframe[fillstyle=solid,fillcolor=boxcol](-9.5,-5.5)(-6.5,-2.5)
    \rput(-8,-4){\tiny $\balance[q]$}
    \psframe[fillstyle=solid,fillcolor=boxcol](-1.5,-1.5)(1.5,1.5)
    \rput(0,0){\tiny $\balance[p]$}
    \psframe[fillstyle=solid,fillcolor=boxcol](-9.5,-6)(-2.5,-9)
    \rput(-6,-7.5){$\psi$}
    \psframe[fillstyle=solid,fillcolor=boxcol](2.5,-10)(9.5,-13)
    \rput(6,-11.5){$\omega$}
    \uput[90](-8,16){$\strut s$}
    \uput[90](0,16){$\strut t$}
    \uput[90](6,16){$\strut p \parr q$}
  \end{pspicture}
  {\quad=\quad}
  \begin{pspicture}[shift=-1.5](-12,-13)(0,19)
    \psline(-2,14)(-2,16)
    \psarc(-2,12){2}{0}{180}
    \psdots(-2,14)
    \psline(-4,12)(-8,8)(-8,4)
    \psdots[linecolor=white](-6,10) 
    \psline(-4,4)(-4,8)(-8,12)(-8,16) 
    \psframe[fillstyle=solid,fillcolor=boxcol](-9.5,4.5)(-6.5,7.5)
    \rput(-8,6){\tiny $\balance[p]$}
    \psframe[fillstyle=solid,fillcolor=boxcol](-9.5,4)(-2.5,1)
    \rput(-6,2.5){$\omega$}
    \psline(0,12)(0,2)(-8,-6)(-8,-10)
    \psdots[linecolor=white](-6,-4) 
    \psline(-4,-10)(-4,-6)(-12,2)(-12,16) 
    \psframe[fillstyle=solid,fillcolor=boxcol](-9.5,-9.5)(-6.5,-6.5)
    \rput(-8,-8){\tiny $\balance[q]$}
    \psframe[fillstyle=solid,fillcolor=boxcol](-9.5,-10)(-2.5,-13)
    \rput(-6,-11.5){$\psi$}
    \uput[90](-12,16){$\strut s$}
    \uput[90](-8,16){$\strut t$}
    \uput[90](-2,16){$\strut p \parr q$}
  \end{pspicture}
\end{center}
\end{proof}

\begin{lem}
The two constructions outlined above are inverse to one another
---that is, $\nunuCycle=\Cycle$ and $\nunubalance=\balance$ 
(where $\nucycle$ corresponds to $\nuCycle$).    
\end{lem}
\begin{proof}
That, for any given $\cycle$, 
$\nunuCycle=\Cycle$ is simply a more rigorous version
of the argument appearing in section \ref{sec:intr}:
one simply substitutes a $=:$ for a $\stackrel?=$.   

The converse requires a little more work: given a $\balance$, 
we must work out the $\nucycle$ corresponding to $\nuCycle$.
According to Remark \ref{rem:upperlowercase}, we have 
$ \nucycle[p] =
\rCurry{\nuCycle[p,\perp{p}](\invlCurry{\id[\perp{p}])}} $.
Hence, $\nunubalance[p]$ equals
\begin{center}
  \begin{pspicture}[shift=-8](0,-4)(8,12)
    \psarc(6,0){2}{180}{0}
    \psline(8,0)(8,10)
    \psline(0,8)(4,4)(4,0)
    \psdots[linecolor=white](2,6) 
    \psarc(2,8){2}{0}{180}
    \psline(4,8)(0,4)(0,-2) 
    \psframe[fillstyle=solid,fillcolor=boxcol](2.5,0.5)(5.5,3.5)
    \rput(4,2){\tiny $\strut \nucycle[p]$}
    \uput[270](0,-2){$\strut p$}
    \uput[90](8,10){$\strut p$}
  \end{pspicture}
  {\quad=\quad} 
  \begin{pspicture}[shift=-9](-4,-7)(12,19)
    \psline(-4,14)(0,10)
    \psdots[linecolor=white](-2,12) 
    \psarc(-2,14){2}{0}{180}
    \psline(0,14)(-4,10)(-4,-4)
    \psline(0,0)(0,4)(4,8)
    \psdots[linecolor=white](2,6)
    \psline(4,0)(4,4)(0,8)(0,10)
    \psframe[fillstyle=solid,fillcolor=boxcol](-1.5,0.5)(1.5,3.5)
    \rput(0,2){\tiny $\balance[p]$}
    \psarc(2,0){2}{180}{0}
    \psarc(6,8){2}{0}{180}
    \psline(8,8)(8,-2)
    \psarc(10,-2){2}{180}{0}
    \psline(12,-2)(12,16)
    \uput[270](-4,-4){$\strut p$}
    \uput[90](12,16){$\strut p$}
    \psframe[linestyle=dotted](-2,-2.5)(10,10.5)
  \end{pspicture}
  {\quad=\quad} 
  \begin{pspicture}[shift=-9](-1.5,-7)(8,19)
    \psline(0,0)(0,4)(0,12)
    \psline(4,0)(4,4) 
    \psframe[fillstyle=solid,fillcolor=boxcol](-1.5,0.5)(1.5,3.5)
    \rput(0,2){\tiny $\balance[p]$}
    \psarc(2,0){2}{180}{0}
    \psarc(6,4){2}{0}{180}
    \psline(8,4)(8,-4)
    \uput[270](8,-4){$\strut p$}
    \uput[90](0,12){$\strut p$}
  \end{pspicture}
  {\quad=\quad} 
  \begin{pspicture}[shift=-9](-1.5,-7)(1.5,19)
    \psline(0,-4)(0,8)
    \psframe[fillstyle=solid,fillcolor=boxcol](-1.5,0.5)(1.5,3.5)
    \rput(0,2){\tiny $\balance[p]$}
    \uput[270](0,-4){$\strut p$}
    \uput[90](0,8){$\strut p$}
  \end{pspicture}
\end{center}
This concludes the proof of Theorem \ref{thm:balance}.
\end{proof}

By similar arguments, it is possible to show, for arbitrary
natural isomorphisms $\balance:\arrow:\Id[\K]->\Id[\K];$, 
that $\balance[\e]=\id[\e]$ if and only if $\nucycle$ satisfies $\tnul$, 
and that $\balance[\d]=\id[\d]$ if and only if $\nucycle$ satisfies
$\pnul$. 
Applying Lemma \ref{lem:mbm}, one sees that $\tsbax$ and
$\balance[\d]=\id[\d]$ (or, alternatively, $\psbax$ and
$\balance[\e]=\id[\e]$) suffice to show that $\balance$ is a balance
for $\K$ in the sense of Definition \ref{def:balance}.

\begin{cor}
$\ioCycle$
is a \cthing\ if and only if $\tbraid$ (and therefore also $\pbraid$)
is a symmetry.      
\end{cor}

We now turn to the question of {\kq}cyclicity; in particular, 
whether it is possible that $\ioCycle$ be a \kq\cthing\ even if
$\tbraid$ is not a symmetry.  
Observe that every object $p$ of a \brdstaut\ category $\K$ admits a
canonical \emph{$4\pi$-twist}: 
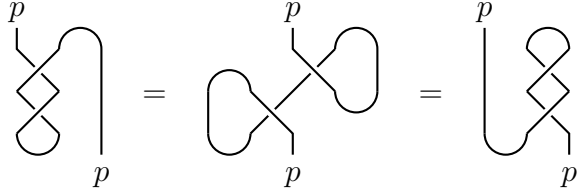
\begin{center}
    \begin{pspicture}[shift=-8](0,-3)(8,15) 
      \psline(0,12)(0,10)(4,6)
      \psline(0,6)(4,2)
      \psdots[linecolor=white](2,8)(2,4)
      \psline(0,2)(4,6)
      \psline(0,6)(4,10)
      \psarc(2,2){2}{180}{0}
      \psarc(6,10){2}{0}{180}
      \psline(8,10)(8,0)
      \uput[90](0,12){$\strut p$} 
      \uput[270](8,0){$\strut p$} 
    \end{pspicture}
    {\quad=\quad}
    \begin{pspicture}[shift=-8](0,-3)(16,15) 
      \psarc(2,2){2}{180}{0}
      \psarc(2,6){2}{0}{180}
      \psarc(14,6){2}{180}{0}
      \psarc(14,10){2}{0}{180}
      \psline(0,2)(0,6)
      \psline(16,10)(16,6)
      \psline(4,2)(12,10)
      \psdots[linecolor=white](10,8)(6,4)
      \psline(4,6)(8,2)(8,0)
      \psline(8,12)(8,10)(12,6)
      \uput[90](8,12){$\strut p$} 
      \uput[270](8,0){$\strut p$} 
    \end{pspicture}
    {\quad=\quad}
    \begin{pspicture}[shift=-8](0,-3)(8,15) 
      \psline(4,10)(8,6)
      \psline(4,6)(8,2)(8,0)
      \psdots[linecolor=white](6,8)(6,4)
      \psline(4,2)(8,6)
      \psline(4,6)(8,10)
      \psarc(2,2){2}{180}{0}
      \psarc(6,10){2}{0}{180}
      \psline(0,12)(0,2)
      \uput[90](0,12){$\strut p$} 
      \uput[270](8,0){$\strut p$} 
    \end{pspicture}
\end{center}
---we shall denote this map $\stitch[p]$. 
Of course, if the braiding happens to be a symmetry, then $\stitch[p]$
will be the identity for all $p$.
The converse is false: 
Gabriella B\"{o}hm and the second author have together constructed
a class of braided Hopf algebras $H$ with the property that $\Rep{H}$,
the category of finite-dimensional $H$-modules, satisfies
$\stitch=\id$ and $\braid^2\neq\id$.     
The simplest of these is the Drinfeld double of (the group algebra of)
$\mathbb Z_2$ together with its universal $R$-matrix.   

\begin{thm} \label{thm:kqbalance}
Using the same notation as before, 
$\nuCycle$ is a \kq\cthing\ if and only if 
\begin{diagram}{ 
  p \ar[r]^-{\balance[p]} 
  & p \ar[r]^-{\canon[p]} 
  & \prep{(\perp{p})} \ar[r]^-{\prep{(\balance[\perp{p}])}} 
  & \prep{(\perp{p})} \ar[r]^-{\canon[p]\inverse} 
  & p }
\end{diagram}
equals $\stitch[p]$ for all $p$.  
In particular, $\ioCycle$ is a \kq\cthing\ if and only if
$\stitch[p]=\id[p]$ for all $p$. 
\end{thm}

Naturally, we shall call a \nt\ $\arrow\balance:\Id[\K]->\Id[\K];$ a 
\emph{{\kq}balance} if it satisfies this condition.
(An equivalent and arguably more elegant condition is that
$\stitch[\perp{p}]=\perp{(\stitch[p])}$ should equal 
\begin{diagram}{ 
  \perp{p} 
  & \perp{p} 
  \ar[l]_-{\perp{(\balance[p])}} 
  & \perp{(\prep{(\perp{p})})} 
  \ar[l]_-{\perp{\canon[p]}} 
  & \perp{(\prep{(\perp{p})})} 
  \ar[l]_-{\perp{(\prep{(\balance[\perp{p}])})}} 
  & \perp{p} 
  \ar@/^/[dl]^-{\id[\perp{p}]}
  \ar[l]_-{\perp{(\canon[p]\inverse)}} 
  \\ 
  && \perp{p} \ar[u]^-{\canon[\perp{p}]} 
  \ar@/^/[ul]^-{\id[\perp{p}]}
  & \perp{p} \ar[u]_-{\canon[\perp{p}]} 
  \ar[l]_-{\balance[\perp{p}]} 
  }
\end{diagram}
for all $p$.) 
\begin{proof}[of Theorem \ref{thm:kqbalance}]
Suppose that $\nuCycle$ is a \kq\cthing; then 
\begin{center}
    \begin{pspicture}[shift=-8](0,-3)(8,15) 
      \psline(0,12)(0,10)(4,6)
      \psline(0,6)(4,2)
      \psdots[linecolor=white](2,8)(2,4)
      \psline(0,2)(4,6)
      \psline(0,6)(4,10)
      \psarc(2,2){2}{180}{0}
      \psarc(6,10){2}{0}{180}
      \psline(8,10)(8,0)
      \uput[90](0,12){$\strut p$} 
      \uput[270](8,0){$\strut p$} 
    \end{pspicture}
    {\quad$\stackrel{\displaystyle\kprime}{=}$\quad}
    \begin{pspicture}[shift=-16](-1.5,-3)(8,31) 
      \psline(4,22)(4,18)(0,14)(0,10)(4,6)
      \psline(0,6)(4,2)
      \psdots[linecolor=white](2,8)(2,4)(2,16)
      \psline(0,2)(4,6)
      \psline(0,6)(4,10)(4,14)(0,18)(0,22)(4,26)
      \psdots[linecolor=white](2,24)
      \psline(4,22)(0,26)(0,28)
      \psarc(2,2){2}{180}{0}
    \psframe[fillstyle=solid,fillcolor=boxcol](-1.5,10.5)(1.5,13.5)
    \rput(0,12){\tiny $\balance[p]$}
    \psframe[fillstyle=solid,fillcolor=boxcol](-1.5,18.5)(1.5,21.5)
    \rput(0,20){\tiny $~\balance[\perp{p}]$}
      \uput[90](0,28){$\strut p$} 
      \psarc(6,26){2}{0}{180}
      \psline(8,26)(8,0)
      \uput[270](8,0){$\strut p$} 
    \end{pspicture}
    {\quad=\quad} 
    \begin{pspicture}[shift=-16](-1.5,-3)(8,31) 
      \psline(4,22)(4,18)(4,6)
      \psline(0,6)(4,2)
      \psdots[linecolor=white](2,4)
      \psline(0,2)(4,6)
      \psline(0,6)(0,18)(0,22)(4,26)
      \psdots[linecolor=white](2,24)
      \psline(4,22)(0,26)(0,28)
      \psarc(2,2){2}{180}{0}
    \psframe[fillstyle=solid,fillcolor=boxcol](2.5,10.5)(5.5,13.5)
    \rput(4,12){\tiny $\balance[p]$}
    \psframe[fillstyle=solid,fillcolor=boxcol](-1.5,18.5)(1.5,21.5)
    \rput(0,20){\tiny $~\balance[\perp{p}]$}
      \uput[90](0,28){$\strut p$} 
      \psarc(6,26){2}{0}{180}
      \psline(8,26)(8,0)
      \uput[270](8,0){$\strut p$} 
    \end{pspicture}
    {\quad=\quad} 
    \begin{pspicture}[shift=-9](-1.5,-3)(5.5,17) 
      \psline(0,2)(0,14)
      \psline(4,2)(4,6)
      \psarc(2,2){2}{180}{0}
    \psframe[fillstyle=solid,fillcolor=boxcol](-1.5,9.5)(1.5,12.5)
    \rput(0,11){\tiny $\balance[p]$}
    \psframe[fillstyle=solid,fillcolor=boxcol](2.5,2.5)(5.5,5.5)
    \rput(4,4){\tiny $~\balance[\perp{p}]$}
      \uput[90](0,14){$\strut p$} 
      \psarc(6,6){2}{0}{180}
      \psline(8,6)(8,0)
      \uput[270](8,0){$\strut p$} 
    \end{pspicture}
\end{center}
as desired.

Conversely, we note that the following are equivalent:
\begin{center}
    \begin{pspicture}[shift=-10](0,-3)(8,15) 
      \psline(0,12)(0,10)(4,6)
      \psline(0,6)(4,2)
      \psdots[linecolor=white](2,8)(2,4)
      \psline(0,2)(4,6)
      \psline(0,6)(4,10)
      \psarc(2,2){2}{180}{0}
      \psarc(6,10){2}{0}{180}
      \psline(8,10)(8,0)
      \uput[90](0,12){$\strut p$} 
      \uput[270](8,0){$\strut p$} 
    \end{pspicture}
    {~=~} 
    \begin{pspicture}[shift=-10](-1.5,-3)(8,15) 
      \psline(0,2)(0,12)
      \psline(4,2)(4,6)
      \psarc(2,2){2}{180}{0}
    \psframe[fillstyle=solid,fillcolor=boxcol](-1.5,8.5)(1.5,11.5)
    \rput(0,10){\tiny $\balance[p]$}
    \psframe[fillstyle=solid,fillcolor=boxcol](2.5,2.5)(5.5,5.5)
    \rput(4,4){\tiny $~\balance[\perp{p}]$}
      \uput[90](0,12){$\strut p$} 
      \psarc(6,6){2}{0}{180}
      \psline(8,6)(8,0)
      \uput[270](8,0){$\strut p$} 
    \end{pspicture}
    {\quad$\iff$\quad}
    \begin{pspicture}[shift=-5](0,0)(4,13) 
      \psline(0,10)(4,6)
      \psline(0,6)(4,2)
      \psdots[linecolor=white](2,8)(2,4)
      \psline(0,2)(4,6)
      \psline(0,6)(4,10)
      \psarc(2,2){2}{180}{0}
      \uput[90](0,10){$\strut p$} 
      \uput[90](4,10){$\strut \perp{p}$} 
    \end{pspicture}
    {~=~} 
    \begin{pspicture}[shift=-5](-1.5,0)(5.5,13) 
      \psline(0,2)(0,10)
      \psline(4,2)(4,10)
      \psarc(2,2){2}{180}{0}
    \psframe[fillstyle=solid,fillcolor=boxcol](-1.5,4.5)(1.5,7.5)
    \rput(0,6){\tiny $\balance[p]$}
    \psframe[fillstyle=solid,fillcolor=boxcol](2.5,4.5)(5.5,7.5)
    \rput(4,6){\tiny $~\balance[\perp{p}]$}
      \uput[90](0,10){$\strut p$} 
      \uput[90](4,10){$\strut \perp{p}$} 
    \end{pspicture}
    {\quad$\iff$\quad}
    \begin{pspicture}[shift=-5](0,0)(4,13) 
      \psline(0,2)(0,10)
      \psline(4,2)(4,10)
      \psarc(2,2){2}{180}{0}
      \uput[90](0,10){$\strut p$} 
      \uput[90](4,10){$\strut \perp{p}$} 
    \end{pspicture}
    {~=~} 
    \begin{pspicture}[shift=-9](0,-4)(4,15) 
      \psline(0,-2)(0,2)(4,6)
      \psline(0,6)(4,10)
      \psdots[linecolor=white](2,8)(2,4)
      \psline(0,10)(4,6)
      \psline(0,6)(4,2)(4,-2)
      \psarc(2,-2){2}{180}{0}
    \psframe[fillstyle=solid,fillcolor=boxcol](-1.5,-1.5)(1.5,1.5)
    \rput(0,0){\tiny $\balance[p]$}
    \psframe[fillstyle=solid,fillcolor=boxcol](2.5,-1.5)(5.5,1.5)
    \rput(4,0){\tiny $~\balance[\perp{p}]$}
      \uput[90](0,10){$\strut p$} 
      \uput[90](4,10){$\strut \perp{p}$} 
    \end{pspicture}
    {\quad$\iff$\quad}
    \begin{pspicture}[shift=-12](0,-7)(0,15) 
      \psline(0,-4)(0,12)
      \uput[270](0,-4){$\strut p$} 
      \uput[90](0,12){$\strut p$} 
    \end{pspicture}
    {~=~} 
    \begin{pspicture}[shift=-12](0,-7)(8,15) 
      \psline(0,-2)(0,2)(4,6)
      \psline(0,6)(4,10)
      \psdots[linecolor=white](2,8)(2,4)
      \psline(0,12)(0,10)(4,6)
      \psline(0,6)(4,2)(4,-2)
      \psarc(2,-2){2}{180}{0}
    \psframe[fillstyle=solid,fillcolor=boxcol](-1.5,-1.5)(1.5,1.5)
    \rput(0,0){\tiny $\balance[p]$}
    \psframe[fillstyle=solid,fillcolor=boxcol](2.5,-1.5)(5.5,1.5)
    \rput(4,0){\tiny $~\balance[\perp{p}]$}
      \uput[90](0,12){$\strut p$} 
      \psarc(6,10){2}{0}{180}
      \psline(8,10)(8,-4)
      \uput[270](8,-4){$\strut p$} 
    \end{pspicture}
\end{center}
Now suppose that $\nubalance$ is a {\kq}balance; then
\begin{center}
\begin{pspicture}[shift=-10](0,-7)(0,29)
\psline(0,-4)(0,26)
\uput[90](0,26){$\strut p$}
\uput[270](0,-4){$\strut p$}
\end{pspicture}
{\quad=\quad}
\begin{pspicture}[shift=-7](0,-4)(24,26)
\psline(8,20)(12,24)
\psline(12,20)(0,8)(0,4)(4,0)
\psline(12,8)(16,4)(16,0)
\psframe[fillstyle=solid,fillcolor=boxcol](14,0.5)(18,3.5)
\rput(16,2){\tiny $\strut\cycle[\perp{p}]$}
\psdots[linecolor=white](2,2)(14,6)(10,18)(10,22)
\psarcn(2,0){2}{0}{180}
\psline(0,0)(4,4)(4,8)
\psframe[fillstyle=solid,fillcolor=boxcol](2.5,4.5)(5.5,7.5)
\rput(4,6){\tiny $\strut\cycle[p]$}
\psarc(6,8){2}{0}{180}
\psline(8,8)(8,-2)
\psarcn(10,-2){2}{0}{180}
\psline(12,-2)(12,4)(16,8)
\psarc(14,8){2}{0}{180}
\psarcn(18,0){2}{0}{180}
\psline(20,0)(20,8)(8,20)
\psline(12,20)(8,24)(8,26)
\psarc(14,24){2}{0}{180}
\psline(16,24)(24,16)(24,-4)
    \psframe[linestyle=dotted](-1.5,-3)(9.5,11)
    \psframe[linestyle=dotted](10.5,-3)(21.5,11)
\uput[90](8,26){$\strut p$}
\uput[270](24,-4){$\strut p$}
\end{pspicture}
{\quad=\quad}
\begin{pspicture}[shift=-7](12,-6)(28,22)
\psline(20,0)(20,8)(12,16)(16,20)
\psline(12,8)(16,4)(16,0)
\psframe[fillstyle=solid,fillcolor=boxcol](14,0.5)(18,3.5)
\rput(16,2){\tiny $\strut\cycle[\perp{p}]$}
\psdots[linecolor=white](14,6)(14,18)
\psline(12,0)(12,4)(16,8)
\psframe[fillstyle=solid,fillcolor=boxcol](10.5,0.5)(13.5,3.5)
\rput(12,2){\tiny $\strut\cycle[p]$}
\psarc(14,8){2}{0}{180}
\psarcn(18,0){2}{0}{180}
\psarcn(18,0){6}{0}{180}
\psline(24,0)(24,8)(12,20)(12,22)
\psarc(18,20){2}{0}{180}
\psline(20,20)(28,12)(28,-6)
\uput[90](12,22){$\strut p$}
\uput[270](28,-6){$\strut p$}
\end{pspicture}
{\quad=\quad}
\begin{pspicture}[shift=-7](16,-4)(24,18)
\psline(16,12)(16,0)
\psframe[fillstyle=solid,fillcolor=boxcol](14,2.5)(18,5.5)
\rput(16,4){\tiny $\strut\cycle[\perp{p}]$}
\psline(20,12)(20,8)
\psframe[fillstyle=solid,fillcolor=boxcol](18.5,8.5)(21.5,11.5)
\rput(20,10){\tiny $\strut\cycle[p]$}
\psarc(18,12){2}{0}{180}
\psarcn(18,0){2}{0}{180}
\psarcn(22,8){2}{0}{180}
\psarc(22,0){2}{0}{180}
\psline(24,0)(24,-4)
\psline(24,8)(24,18)
\uput[90](24,18){$\strut p$}
\uput[270](24,-4){$\strut p$}
\end{pspicture}
{\quad=\quad}
\begin{pspicture}[shift=-7](13.5,-4)(18.5,18)
\psline(16,-4)(16,18)
\psframe[fillstyle=solid,fillcolor=boxcol](14,2.5)(18,5.5)
\rput(16,4){\tiny $\strut\cycle[\perp{p}]$}
\psframe[fillstyle=solid,fillcolor=boxcol](13.5,8.5)(18.5,11.5)
\rput(16,10){\tiny $\perp{(\cycle[p])}$}
\uput[90](16,18){$\strut p$}
\uput[270](16,-4){$\strut p$}
\end{pspicture}
\end{center}
\end{proof}

Finally we note that, since every \cthing\ is a \kq\cthing, every
balance (on a braided \staut\ category) is a {\kq}balance;
from this one derives the following.  

\begin{cor}
A balance $\balance$ for a \staut\ category $\K$ satisfies 
$\balance[\perp{p}]=\perp{(\balance[p])}$ if and only if it satisfies
$\stitch[p]=\balance[p]^2$. 
\end{cor}

This very important property will be further discussed
in \cite{EggMcC:cld}.

\section{Strictification} \label{sec:zang}

To conclude, we address 
the issue of strictifying negation in \staut\ categories,
in both cyclic and arbitrary cases. 
It seems otiose to rigorously state and prove theorems 
(which would entail, among other things, 
fully written-out definitions of \emph{morphism} of \cycstaut\
categories, and of \emph{two-cell} between such morphisms)   
when the truth of what we assert is so manifest.  
Consequently, we proceed in a slightly less formal fashion than 
heretofore; 
our techniques are simple and obvious extensions of those used
in \cite{CHS}.  

\begin{defs} \label{def:zang}
  Let $\K=\Kexplained$ be an \arbstaut\ category; 
  then, by a \emph{\zang\ of (linear) adjoints} 
  we mean a $\Z$-indexed family of $\K$-objects, 
  $p=\familyof[n\in\Z]{p_n}$, together with $\K$-arrows  
  \begin{diagram}{ 
    \e \ar[r]^-{\tau_n} & p_{n+1} \parr p_n 
    & p_n \tens p_{n+1} \ar[r]^-{\gamma_n} & \d } 
  \end{diagram}
  satisfying the (linear) triangle identities of \cite{CocSee:ldc}; 
  and, by the \emph{canonical} \zang\ of adjoints determined by a
  $\K$-object $p$, we mean the family
  $\zangify{p}=\familyof[n\in\Z]{\zangify{p}_n}$ given by  
  \[ \zangify{p}_n := \left\{ 
  \begin{array}{cl}
    p^{* \cdots *} & \textrm{if $n$ is positive} \\
    p & \textrm{if $n$ is zero} \\
    {}^{* \cdots *}p & \textrm{if $n$ is negative} 
  \end{array}
  \right. \]
  together with the canonical ({chosen}) linear adjunctions between
  $\zangify{p}_n$ and $\zangify{p}_{n+1}$.  

  Similarly, by a \emph{\zang\ of (linear) mates}  
  $\arrow:p=\familyof[n\in\Z]{p_n}->\familyof[n\in\Z]{q_n}=q;$
  we mean a $\Z$-indexed family of $\K$-arrows 
  $\omega=\familyof[n\in\Z]{\omega_n}$ with 
  $\omega_n \in \xhom[\K]{p_n}{q_n}$ if $n$ is even, and 
  $\omega_n \in \xhom[\K]{q_n}{p_n}$ if $n$ is odd, all 
  satisfying the (linear) mateship relations of \cite{CocSee:ldc}; 
  and by the \emph{canonical} \zang\ of mates determined by a
  $\K$-arrow $\omega$, we mean the family 
  $\zangify{\omega}=\familyof[n\in\Z]{\zangify{\omega}_n}$ given by 
  \[ \zangify{\omega}_n := \left\{ 
  \begin{array}{cl}
    \omega^{* \cdots *} & \textrm{if $n$ is positive} \\
    \omega & \textrm{if $n$ is zero} \\
    {}^{* \cdots *}\omega & \textrm{if $n$ is negative} 
  \end{array}
  \right. \]

  The category of all {\zang}s of adjoints, with {\zang}s of 
  mates between them, will be denoted $\Zang\K$.
\end{defs}

It is evident that $\Zang\K$ carries a \staut\ structure, given
(in part) by   
\[
\begin{array}{rclrcl}
  (p \tens q)_n &=& \left\{ 
  \begin{array}{ll}
    p_n \tens q_n & \textrm{if $n$ is even} \\
    q_n \parr p_n & \textrm{if $n$ is odd} 
  \end{array}
  \right. 
  & (p \parr q)_n &=& \left\{ 
  \begin{array}{ll}
    p_n \parr q_n & \textrm{if $n$ is even} \\
    q_n \tens p_n & \textrm{if $n$ is odd} 
  \end{array}
  \right. 
  \\ \strut \\ 
  \e_n &=& \left\{ 
  \begin{array}{ll}
    \e & \textrm{if $n$ is even} \\
    \d & \textrm{if $n$ is odd} 
  \end{array}
  \right. 
  & \d_n &=& \left\{ 
  \begin{array}{ll}
    \d & \textrm{if $n$ is even} \\
    \e & \textrm{if $n$ is odd} 
  \end{array}
  \right.
  \\ \strut \\
  (\perp{p})_n &=& p_{n+1}
  & (\prep{p})_n &=& p_{n-1} 
\end{array}
\]
and that this has \emph{strict negations} in the sense described in
section \ref{sec:intr}.   
Moreover, $\zangify{\blank}$ and  $\blank_0$ define an (adjoint)
equivalence of \staut\ categories between $\K$ and $\Zang\K$.     
Hence every \staut\ category is equivalent to one with strict
negations. 
It follows that any \cthing\ $\cycle$ on $\K$ can be extended to a
\cthing\ $\zangcycle$ on $\Zang\K$;
explicitly, 
\begin{diagram}{ 
(\perp{p})_n = p_{n+1} \ar[r]^-{\sim}
& \perp{(p_n)} \ar[r]^-{\cycle[p_n]} 
& \prep{(p_n)} \ar[r]^-{\sim}
& p_{n-1} = (\prep{p})_n } 
\end{diagram}
is the $n$th component of $\zangcycle_p$.

(Note also that the two monoidal structures of $\Zang\K$ are strict
if and only if the same is true of $\K$.  
Hence, to produce a fully strict \staut\ category equivalent to $\K$, 
one could first strictify its \ld\ structure, and then apply the
$\Zang-$ construction.) 

\begin{defn} \label{def:fang}
  Let $(\K,\cycle)$ be a \cycstaut\ category; 
  then, by a \emph{\fang\ of (linear) adjoints}, we mean a \zang\ of
  adjoints $p=\familyof[n\in\Z]{p_n}$ satisfying  
  \[ p_{n+1}=p_{n-1} \qquad \textrm{and} \qquad 
  \gamma_{n}=\Cycle[p_{n-1},p_{n}](\gamma_{n-1}) \]
  for all $n \in \Z$.  
  The full subcategory of $\Zang\K$ determined by the {\fang}s will be 
  denoted $\Fang\K$.   
\end{defn}

\begin{lem} \label{lem:zang}
  $\Fang\K$ is a sub-\staut\ category of $\Zang\K$---that
  is, the class of {\fang}s is closed under $\tens$, $\parr$, $\e$,
  $\d$, $\perp\blank$ and $\prep\blank$.
  Moreover, the restriction of $\zangcycle$ to $\Fang\K$ is the
  identity. 
\end{lem}
\begin{proof}
  For the first statement, we prove only that if $p$ and $q$ are
  {\fang}s, then so is $p \tens q$.  
  It is trivial that $(p \tens q)_{n+1}=(p \tens q)_{n-1}$, 
  but it is non-trivial that the other condition still holds.  
  The full definition of $p \tens q$ (which was only partially 
  described above) includes the following.  
  \begin{diagram}{ 
    (p \tens q)_{2m-1} \tens (p \tens q)_{2m} 
    \ar@<2pc>[dd]^-{\id} 
    \ar@<-2pc>[d]_-{\strut\gamma_{2m-1}}^-{:=\strut} 
    & (p \tens q)_{2m} \tens (p \tens q)_{2m+1} 
    \ar@<2pc>[dd]^-{\id} 
    \ar@<-2pc>[d]_-{\strut\gamma_{2m}}^-{:=\strut} 
    \\ {\d\hspace{4pc}} & {\d\hspace{4pc}} 
    \\ (q_{2m-1} \parr p_{2m-1}) \tens (p_{2m} \tens q_{2m}) 
    \ar@<2pc>[u]^-{\lbind(\gamma_{2m-1},\gamma_{2m-1})} 
    & (p_{2m} \tens q_{2m}) \tens (q_{2m+1} \parr p_{2m+1}) 
    \ar@<2pc>[u]^-{\rbind(\gamma_{2m},\gamma_{2m})} }
  \end{diagram} 
  Hence, \mtwo\ entails 
  \begin{eqnarray*}
    \gamma_{2m} &=& \rbind(\gamma_{2m},\gamma_{2m}) 
    \\ &=& \rbind(\Cycle[p_{2m-1},p_{2m}](\gamma_{2m-1}),
    \Cycle[q_{2m-1},q_{2m}](\gamma_{2m-1})) 
    \\ &=& \Cycle[q_{2m-1} \parr p_{2m-1},p_{2m} \tens q_{2m}]
    (\lbind(\gamma_{2m-1},\gamma_{2m-1})) 
    \\ &=& \Cycle[(p \tens q)_{2m-1},(p \tens q)_{2m}](\gamma_{2m-1}) 
  \end{eqnarray*}
  which covers the case $n=2m$;
  the case $n=2m+1$ follows by a symmetric argument
  involving \varmtwo\ or, alternatively, by 
  invoking \kprime.  

  The second statement is proven as follows.
  \begin{center}
  \begin{pspicture}[shift=-11](0,-3)(8,19) 
    \psarcn(2,2){2}{0}{180}
    \psarc(2,14){2}{0}{180}
    \psline(0,2)(0,14)
    \psline(4,2)(6,4)(8,2)(8,0)  
    \psline(4,14)(6,12)(8,14)(8,16)  
    \psframe[fillstyle=solid,fillcolor=boxcol](4.5,10.5)(7.5,13.5)
    \rput(6,12){$\gamma$} 
    \psframe[fillstyle=solid,fillcolor=boxcol](4.5,2.5)(7.5,5.5)
    \rput(6,4){$\tau$}    
    \psframe[fillstyle=solid,fillcolor=boxcol](-1.5,6.5)(1.5,9.5)
    \rput(0,8){$\cycle$}  
    \uput[90](8,16){$\strut p_{n+1}$}
    \uput[270](8,0){$\strut p_{n-1}$}
  \end{pspicture}
  {~=~} 
  \begin{pspicture}[shift=-8](-2,-2)(16,14) 
    \psarcn(6,4){6}{0}{180}
    \psarc(2,8){2}{0}{180}
    \psline(0,4)(0,8)
    \psline(4,8)(6,6)(8,8)(8,14)  
    \psline(12,4)(12,10)(14,12)(16,10)(16,-2)
    \psframe[fillstyle=solid,fillcolor=boxcol](4.5,4.5)(7.5,7.5)
    \rput(6,6){$\gamma$} 
    \psframe[fillstyle=solid,fillcolor=boxcol](12.5,10.5)(15.5,13.5)
    \rput(14,12){$\tau$} 
    \psframe[fillstyle=solid,fillcolor=boxcol](-1.5,4.5)(1.5,7.5)
    \rput(0,6){$\cycle$} 
    \uput[90](8,14){$\strut p_{n+1}$}
    \uput[270](16,-2){$\strut p_{n-1}$}
    \psframe[linestyle=dotted](-2,-3)(13,11)
  \end{pspicture}
  {~=~} 
  \begin{pspicture}[shift=-8](4,1)(16,17) 
    \psline(4,17)(4,12)(8,8)(12,12)(14,14)(16,12)(16,1)
    \psframe[fillstyle=solid,fillcolor=boxcol](4.5,6.5)(11.5,11.5)
    \rput(8,9){$\Cycle(\gamma)$} 
    \psframe[fillstyle=solid,fillcolor=boxcol](12.5,12.5)(15.5,15.5)
    \rput(14,14){$\tau$} 
    \uput[90](4,17){$\strut p_{n+1}$}
    \uput[270](16,1){$\strut p_{n-1}$}
  \end{pspicture}
  {~=~} 
  \begin{pspicture}[shift=-8](8,2)(16,18) 
    \psline(8,18)(8,12)(10,10)(12,12)(14,14)(16,12)(16,2)
    \psframe[fillstyle=solid,fillcolor=boxcol](8.5,8.5)(11.5,11.5)
    \rput(10,10){$\gamma$} 
    \psframe[fillstyle=solid,fillcolor=boxcol](12.5,12.5)(15.5,15.5)
    \rput(14,14){$\tau$}   
    \uput[90](8,18){$\strut p_{n-1}$}
    \uput[270](16,2){$\strut p_{n-1}$}
  \end{pspicture}
  {~=~} 
  \begin{pspicture}[shift=-8](-2,0)(2,16)
    \psline(0,0)(0,16)
    \uput[90](0,16){$\strut p_{n-1}$}
    \uput[270](0,0){$\strut p_{n-1}$}
  \end{pspicture}
  \end{center}
\end{proof}

Hence, every \cycstaut\ category is equivalent to one which has 
\emph{a strict negation} in the sense described in
section \ref{sec:intr}.   

\appendix 
\section{Miscellaneous proofs} 
\psset{unit=3pt}

\begin{lem} \label{appx:snis}
Let $(M,\mu,\unmu)$ and $(N,\nu,\unnu)$ be strong monoidal functors
$\arrow:(\J,{\jnewt},\j)->(\K,{\knewt},\k);$, 
and $\omega$ be a natural isomorphism 
$\arrow:M->N;$ 
satsifying 
\begin{diagram}{
  M(p) \knewt M(q) \ar[r]^-{\mu} 
  \ar[d]_-{\omega_p\knewt\omega_q}
  & M(p \jnewt q) 
  \ar[d]^-{\omega_{p \jnewt q}}
  \\ N(p) \knewt N(q) \ar[r]^-{\nu} 
  & N(p \jnewt q) }
\end{diagram}
---then 
\begin{inline}{ 
  \k \ar[r]^-{\unmu} \fixwonky{d}{rr}{_-{\unnu}}
  & M(\j) \ar[r]^-{\omega_\j} & N(\j) }
\end{inline}
also holds.
\end{lem}
\begin{proof}

It suffices to show that $\unmu;\omega_j;\unnu\inverse$ is idempotent
\wrt\ composition (since it is already known to be invertible).
But, by the Eckmann-Hilton argument, it is equivalent to show that it
is idempotent \wrt\ tensor (which is to say, in the arrow category,  
\[ (\unmu;\omega_j;\unnu\inverse)\knewt(\unmu;\omega_j;\unnu\inverse)
\iso (\unmu;\omega_j;\unnu\inverse)\] 
via the canonical isomorphism $\arrow\lulaw=\rulaw:\k\knewt\k->\k;$).  
This is demonstrated in the diagram below: 
\begin{diagram}{
    \k \knewt \k 
    \ar[d]_-{\lulaw = \rulaw}
    \ar[r]^-{\id \knewt \unmu}
    \fixwonky{u}{rr}{^-{\unmu \knewt \unmu}}
    \ar@{}[dr]|-{\naty}
    & \k \knewt M(\j) 
    \ar[d]^-{\lulaw}
    \ar[r]^-{\unmu \knewt \id}
    & M(\j) \knewt M(\j)
    \ar[d]^-{\mu}
    \wonky{r}{ddd}{^-{\omega_j \knewt \omega_j}}
    \\ \k 
    \ar[r]_-{\unmu} 
    & M(\j) 
    \ar[d]_-{\omega_j}
    \ar@{}[dr]|-{\naty}
    & M(\j \jnewt \j) 
    \ar[d]^-{\omega_{j \jnewt j}}
    \ar[l]|-{M(\rulaw=\lulaw)}
    \\ \k 
    \ar[d]_-{\lulaw\inverse = \rulaw\inverse}
    & N(\j) 
    \ar[d]^-{\rulaw\inverse}
    \ar[l]_-{\unnu\inverse}
    \ar@{}[dl]|-{\naty}
    & N(\j \knewt \j)
    \ar[d]^-{\nu\inverse}
    \ar[l]|-{N(\rulaw=\lulaw)}
    \\ \k \knewt \k 
    & N(\j) \knewt \k
    \ar[l]^-{\unnu\inverse \knewt \id}
    & N(\j) \jnewt N(\j) 
    \ar[l]^-{\id \knewt \unnu\inverse}
    \fixwonky{d}{ll}{^-{\unnu\inverse \knewt \unnu\inverse}} }
\end{diagram}
the cells labelled $\naty$ are naturality squares; 
the rightmost cell is a special case of the hypothesis; 
the topmost and bottommost cells are trivial; 
the remaining two squares are special cases of the monoidality of $M$
and $N$; 
the outermost cell is the thing being proven. 
\end{proof}


\def\vomega{\omega'}

\begin{lem} \label{appx:base}
  In any \ld\ category, 
  \[ \assoc[q,s \parr t,p]\inverse\fo
      (\llcad[q,s,t]\fo\psi\parr\id[t]\fo\lulaw[t])\tens\id[p]\fo
      \vomega 
      = \id[q]\tens(\rrcad[s,t,p]\fo\id[s]\parr\vomega 
      \fo\rulaw[s])\fo\psi \]
  holds for all $\arrow\psi:q \tens s->\d;$ and 
  $\arrow\vomega:t \tens p->\d;$. 
\end{lem}
\begin{proof}
In the diagram below: 
the pentagon and the two outermost triangles are among the axioms of a 
\ld\ category;  
the central triangles are tautologies; 
the squares are all naturality squares.  
\begin{diagram}@C=1pc@R=3pc{ 
& (q \tens (s \parr t)) \tens p 
\ar[rr]^-{\assoc[q,(s \parr t),p]}
\ar[d]_-{\llcad[q,s,t]\tens\id[p]}
&& q \tens ((s \parr t) \tens p)
\ar[d]^-{\id[q]\tens\rrcad[s,t,p]} 
\\ & ((q \tens s) \parr t) \tens p 
\ar[dl]_-{{(\psi\parr\id[t])\tens\id[p]}~~}
\ar[dr]^-{\rrcad[(q \tens s),t,p]}
&& q \tens (s \parr (t \tens p))
\ar[dl]_-{\llcad[q,s,(t \tens p)]}
\ar[dr]^-{{~\id[q]\tens(\id[s]\parr\vomega)}}
\\ (\d \parr t) \tens p 
\ar[dr]^-{\rrcad[\d,t,p]}
\ar[dd]_-{\lulaw[t]\tens\id[p]}
&& (q \tens s) \parr (t \tens p) 
\ar[dl]_-{{\psi\parr\id[t \tens p]~}} 
\ar[dd]|-{\psi\parr\vomega}
\ar[dr]^-{\id[q \tens s]\parr\vomega}
&& q \tens (s \parr \d) 
\ar[dl]_-{\llcad[q,s,\d]}
\ar[dd]^-{{~\id[q]\tens\rulaw[s]}}
\\ & \d \parr (t \tens p)
\ar[dr]_-{{\id[\d]\parr\vomega}~}
\ar[dl]^-{\lulaw[t \tens p]}
&& (q \tens s) \parr \d 
\ar[dl]^-{\psi\parr\id[\d]}
\ar[dr]_-{\rulaw[q \tens s]}
\\ t \tens p 
\ar[drr]_-{\vomega}
&& \d \parr \d 
\ar[d]|-{\lulaw[\d]=\rulaw[\d]}
&& q \tens s
\ar[dll]^-{\psi}
\\ && \d }
\end{diagram}
\end{proof}

\begin{lem} \label{appx:strings} 
  The two arrows
  $\arrow:\perp{(\xhom[f]rq)}\tens\xhom[\a]rs->\perp{(\xhom[f]sq)};$  
  described in the proof of Theorem \ref{thm:main} are equal.
\end{lem}
\begin{proof}
  We argue (semi)graphically: the two arrows in question are 
  
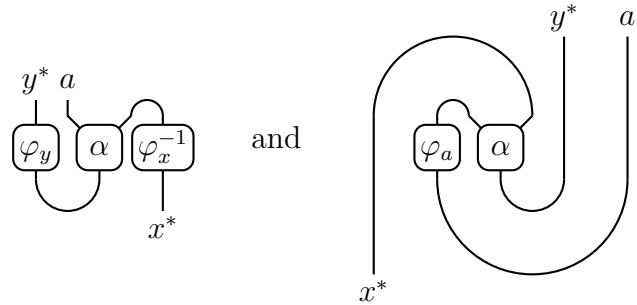
\begin{center}
  \begin{pspicture}[shift=-8](-10,-8)(10,6)
    \uput[90](-8,6){$\strut\perp{y}$}
    \psline(-8,6)(-8,-4) 
    \psarc(-4,-4){4}{180}{0} 
    \psline(0,0)(0,-4) 
    \uput[90](-4,6){$\strut a$}
    \psline(-4,6)(-4,4)(0,0)(4,4)
    \psarc(6,4){2}{0}{180}
    \psline(8,4)(8,-8)
    \uput[270](8,-8){$\strut\perp{x}$}
    \psframe[fillstyle=solid,fillcolor=boxcol](-3,-3)(3,3)
    \rput(0,0){$\strut\alpha$}
    \psframe[fillstyle=solid,fillcolor=boxcol](-11,-3)(-5,3)
    \rput(-8,0){$\strut\cycle[y]$}
    \psframe[fillstyle=solid,fillcolor=boxcol](4,-3)(12,3)
    \rput(8,0){$\strut\cycle[x]\inverse$}
  \end{pspicture} 
  \qquad and \qquad 
  \begin{pspicture}[shift=-19](-16,-19)(16,17)
    \uput[90](8,14){$\strut \perp{y}$}
    \psline(8,14)(8,-4)
    \psarc(4,-4){4}{180}{0} 
    \psline(0,0)(0,-4) 
    \psline(-4,4)(0,0)(4,4)
    \psarc(-6,4){10}{0}{180}
    \psline(-16,4)(-16,-16)
    \uput[270](-16,-16){$\strut\perp{x}$}
    \psarc(-6,4){2}{0}{180}
    \psline(-8,4)(-8,-4)
    \psarc(4,-4){12}{180}{0} 
    \uput[90](16,14){$\strut a$}
    \psline(16,14)(16,-4)
    \psframe[fillstyle=solid,fillcolor=boxcol](-3,-3)(3,3)
    \rput(0,0){$\strut\alpha$}
    \psframe[fillstyle=solid,fillcolor=boxcol](-11,-3)(-5,3)
    \rput(-8,0){$\strut\cycle[a]$}
  \end{pspicture}
  \end{center}
  where $a=\xhom[\a]rs$, $x=\xhom[f]sq$, $y=\xhom[f]rq$, 
  $\alpha$ denotes the original action $\arrow:a \tens x->y;$.

  Clearly, these are equal if and only if 
  
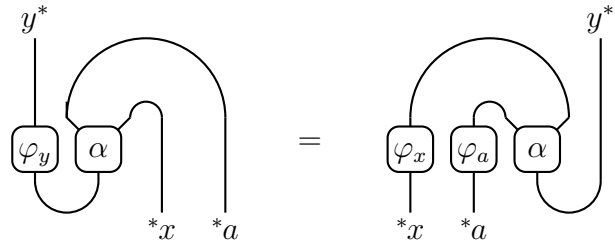
\begin{center}
  \begin{pspicture}[shift=-8](-10,-8)(16,14)
    \uput[90](-8,14){$\strut\perp{y}$}
    \psline(-8,14)(-8,-4) 
    \psarc(-4,-4){4}{180}{0} 
    \psline(0,0)(0,-4) 
    \uput[270](16,-8){$\strut\prep{a}$}
    \psline(16,4)(16,-8)
    \psarc(6,4){10}{0}{180}
    \psline(-4,6)(-4,4)(0,0)(4,4)
    \psarc(6,4){2}{0}{180}
    \psline(8,4)(8,-8)
    \uput[270](8,-8){$\strut\prep{x}$}
    \psframe[fillstyle=solid,fillcolor=boxcol](-3,-3)(3,3)
    \rput(0,0){$\strut\alpha$}
    \psframe[fillstyle=solid,fillcolor=boxcol](-11,-3)(-5,3)
    \rput(-8,0){$\strut\cycle[y]$}
  \end{pspicture} 
  \qquad = \qquad 
  \begin{pspicture}[shift=-11](-18,-11)(8,17)
    \uput[90](8,14){$\strut \perp{y}$}
    \psline(8,14)(8,-4)
    \psarc(4,-4){4}{180}{0} 
    \psline(0,0)(0,-4) 
    \psline(-4,4)(0,0)(4,4)
    \psarc(-6,4){10}{0}{180}
    \psline(-16,4)(-16,-8)
    \psframe[fillstyle=solid,fillcolor=boxcol](-19,-3)(-13,3)
    \rput(-16,0){$\strut\cycle[x]$}
    \uput[270](-16,-8){$\strut\prep{x}$}
    \psarc(-6,4){2}{0}{180}
    \psline(-8,4)(-8,-8)
    \uput[270](-8,-8){$\strut \prep{a}$}
    \psframe[fillstyle=solid,fillcolor=boxcol](-3,-3)(3,3)
    \rput(0,0){$\strut\alpha$}
    \psframe[fillstyle=solid,fillcolor=boxcol](-11,-3)(-5,3)
    \rput(-8,0){$\strut\cycle[a]$}
  \end{pspicture}
  \end{center}
---but here the left-hand side represents 
\begin{diagram}{ 
  \perp{y} \ar[r]^-{\cycle[y]} 
  & \prep{y} \ar[r]^-{\prep\alpha} 
  & \prep{(a \tens x)} \ar[r]^-{\demorgan} 
  & \prep{x} \parr \prep{a} }
\end{diagram}
which, by the naturality of $\cycle$, equals 
\begin{diagram}{ 
  \perp{y} \ar[r]^-{\prep\alpha} 
  & \perp{(a \tens x)} \ar[r]^-{\cycle[a \tens x]} 
  & \prep{(a \tens x)} \ar[r]^-{\demorgan} 
  & \prep{x} \parr \prep{a} }
\end{diagram}
which in turn, by \tbin, equals 
\begin{diagram}{ 
  \perp{y} \ar[r]^-{\prep\alpha} 
  & \perp{(a \tens x)} \ar[r]^-{\demorgan} 
  & \perp{x} \parr \perp{a} \ar[r]^-{\cycle[x] \parr \cycle[a]} 
  & \prep{x} \parr \prep{a} }
\end{diagram}
which is what the right-hand side represents.  
\end{proof}




\begin{lem} \label{appx:nplds}
  In any braided \staut\ category, the following diagrams commute. 
  \begin{diagram}@C=9pt{
      & (p \parr q) \tens r 
      \ar[dl]_-{\pbraid[p,q]\inverse\tens\id[r]}
      \ar[dr]^-{\tbraid[p \parr q,r]}
      &&& r \tens (q \parr p) 
      \ar[dl]_-{\tbraid[q \parr p,r]}
      \ar[dr]^-{\id[r]\tens\pbraid[p,q]\inverse}
      \\ (q \parr p) \tens r 
      \ar[ddd]_-{\rrcad[q,p,r]}
      \ar@{<.}[dr]|-{\tbraid[q \parr p,r]\inverse}
      && r \tens (p \parr q) 
      \ar[ddd]^-{\llcad[r,p,q]}
      \ar@{<.}[dl]|-{\id[r]\tens\pbraid[p,q]}
      & (q \parr p) \tens r 
      \ar[ddd]_-{\rrcad[q,p,r]}
      \ar@{<.}[dr]|-{\pbraid[p,q]\tens\id[r]}
      && r \tens (p \parr q) 
      \ar[ddd]^-{\llcad[r,p,q]}
      \ar@{<.}[dl]|-{\tbraid[p \parr q,r]\inverse}
      \\ & r \tens (q \parr p) 
      &&& (p \parr q) \tens r 
      \\ & q \parr (r \tens p)
      \ar@{<.}[dl]|-{\id[q]\parr\tbraid[r,p]} 
      \ar@{<.}[dr]|-{\pbraid[q,r \tens p]\inverse} 
      &&& (p \tens r) \parr q 
      \ar@{<.}[dl]|-{\pbraid[q,p \tens r]\inverse} 
      \ar@{<.}[dr]|-{\tbraid[r,p]\parr\id[q]} 
      \\ q \parr (p \tens r) 
      \ar[dr]_-{\pbraid[q,p \tens r]} 
      && (r \tens p) \parr q 
      \ar[dl]^-{\tbraid[r,p]\inverse\parr\id[q]} 
      & q \parr (p \tens r) 
      \ar[dr]_-{\id[q]\parr\tbraid[r,p]\inverse} 
      && (r \tens p) \parr q 
      \ar[dl]^-{\pbraid[q,r \tens p]} 
      \\ & (p \tens r) \parr q 
      &&& q \parr (r \tens p) }
  \end{diagram}
\end{lem}
\begin{proof}
  We prove only the first solid diagram;
  the second solid diagram follows by a symmetric argument, 
  and the two dotted diagrams are obtained by attaching naturality
  squares. 

  We work in the graphical calculus for \emph{planar} \staut\
  categories; a box labelled with a plus sign denotes a component of
  $\tbraid$; a box labelled with a minus sign denotes a component of
  $\tbraid\inverse$; the components of $\pbraid$ and $\pbraid\inverse$
  are defined by   
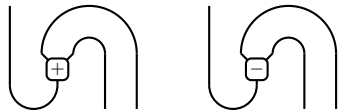
\begin{center}
  \begin{pspicture}(0,1)(16,14)
    \psline(0,14)(0,4)
    \psarcn(3,4){3}{0}{180}
    \psline(6,4)(6,6)
    \psline(4,8)(6,6)(8,8)
    \psarc(10,8){2}{0}{180}
    \psarc(10,8){6}{0}{180}
    \psline(12,8)(12,1)
    \psline(16,8)(16,1)
    \psframe[fillstyle=solid,fillcolor=boxcol](4.5,4.5)(7.5,7.5)
    \rput(6,6){\tiny $+$}
  \end{pspicture}
  \qquad
  \begin{pspicture}(0,1)(16,14)
    \psline(0,14)(0,4)
    \psarcn(3,4){3}{0}{180}
    \psline(6,4)(6,6)
    \psline(4,8)(6,6)(8,8)
    \psarc(10,8){2}{0}{180}
    \psarc(10,8){6}{0}{180}
    \psline(12,8)(12,1)
    \psline(16,8)(16,1)
    \psframe[fillstyle=solid,fillcolor=boxcol](4.5,4.5)(7.5,7.5)
    \rput(6,6){\tiny $-$}
  \end{pspicture}
\end{center}
respectively; \emph{switching links} are denoted by moustaches, and 
\emph{non-switching links} by bullets; what we need to show is
summarised below.   
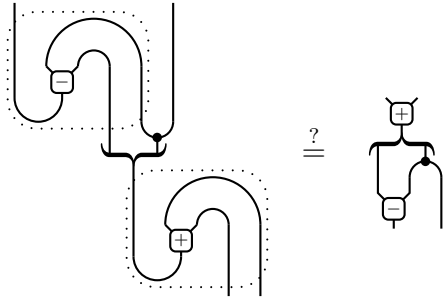
\begin{center}
  \begin{pspicture}[shift=-17](0,-1)(31,36)
    \psline(0,36)(0,24)
    \psarcn(3,24){3}{0}{180}
    \psline(6,24)(6,26)
    \psline(4,28)(6,26)(8,28)
    \psarc(10,28){2}{0}{180}
    \psarc(10,28){6}{0}{180}
    \psline(12,28)(12,17)
    \psline(16,28)(16,21)
    \psframe[fillstyle=solid,fillcolor=boxcol](4.5,24.5)(7.5,27.5)
    \psframe[linestyle=dotted](-1,20)(17,35)
    \rput(6,26){\tiny $-$}
    \psline(15,16)(15,4)
    \psarcn(18,4){3}{0}{180}
    \psline(21,4)(21,6)
    \psline(19,8)(21,6)(23,8)
    \psarc(25,8){2}{0}{180}
    \psarc(25,8){6}{0}{180}
    \psline(27,8)(27,-1)
    \psline(31,8)(31,-1)
    \psframe[fillstyle=solid,fillcolor=boxcol](19.5,4.5)(22.5,7.5)
    \rput(21,6){\tiny $+$}
    \psarcn(18,21){2}{0}{180}
    \psdots(18,19)
    \psline(18,19)(18,17)
    \psline(20,21)(20,36)
    \rput{270}(15,17){\Huge $\left.\right\}$}
    \psframe[linestyle=dotted](14,0)(32,15)
  \end{pspicture}
  {\quad $\stackrel?=$ \quad}
  \begin{pspicture}[shift=-9](-1,3)(7,20) 
    \psline(2,15)(2,18)(4,20)
    \psline(2,18)(0,20)
    \rput{90}(2,14){\Huge $\left.\right\}$}
    \psline(5,14)(5,12)
    \psdots(5,12)
    \psarc(5,10){2}{0}{180}
    \psline(3,10)(3,8)(1,6)(-1,8)(-1,14)
    \psline(1,6)(1,3.5)
    \psframe[fillstyle=solid,fillcolor=boxcol](3.5,16.5)(0.5,19.5)
    \rput(2,18){\tiny $+$}    
    \psframe[fillstyle=solid,fillcolor=boxcol](-0.5,4.5)(2.5,7.5)
    \rput(1,6){\tiny $-$}
    \psline(7,10)(7,3.5)
  \end{pspicture}
\end{center}

By naturality and (a variant of) the braiding axiom, we obtain 
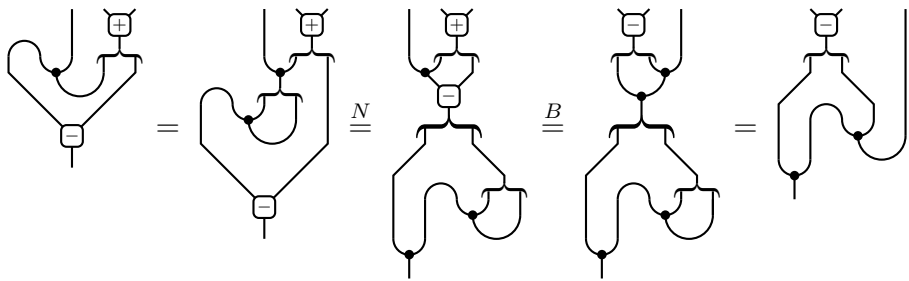
\begin{center}
  \begin{pspicture}[shift=-4](0,0)(16,20)
    \psline(8,0)(8,4)(0,12)(0,14)
    \psline(8,4)(16,12)(16,14)
    \psarc(2,14){2}{0}{180}
    \psline(8,14)(8,20)
    \psarcn(6,14){2}{0}{180}
    \psarcn(9,12){3}{0}{180}
    \psline(12,12)(12,14)
    \rput{90}(14,14){\Large $\left.\right\}$}
    \psline(14,15)(14,18)(12,20)
    \psline(14,18)(16,20)
    \psdots(6,12)
    \psframe[fillstyle=solid,fillcolor=boxcol](6.5,2.5)(9.5,5.5)
    \rput(8,4){\tiny $-$}
    \psframe[fillstyle=solid,fillcolor=boxcol](12.5,16.5)(15.5,19.5)
    \rput(14,18){\tiny $+$}    
  \end{pspicture}
  {~$=$~}
  \begin{pspicture}[shift=-13](0,-3)(16,26)
    \psline(8,-3)(8,1)(0,9)(0,14)
    \psline(8,1)(16,9)(16,20)
    \psarc(2,14){2}{0}{180}
    \psarcn(6,14){2}{0}{180}
    \psarcn(9,12){3}{0}{180}
    \psline(8,14)(8,15)
    \psline(12,12)(12,15)
    \rput{90}(10,15){\Large $\left.\right\}$}
    \psline(10,16)(10,18)
    \psarcn(10,20){2}{0}{180}
    \psline(8,20)(8,26)
    \rput{90}(14,20){\Large $\left.\right\}$}
    \psline(14,21)(14,24)(12,26)
    \psline(14,24)(16,26)
    \psdots(6,12)(10,18)
    \psframe[fillstyle=solid,fillcolor=boxcol](6.5,-0.5)(9.5,2.5)
    \rput(8,1){\tiny $-$}
    \psframe[fillstyle=solid,fillcolor=boxcol](12.5,22.5)(15.5,25.5)
    \rput(14,24){\tiny $+$}    
  \end{pspicture}
  {~$\stackrel{N}=$~}
  \begin{pspicture}[shift=-18](6,-6)(22,28)
    \psline(8,-3)(8,-6)
    \psdots(8,-3)
    \psarcn(8,-1){2}{0}{180}
    \psline(10,4)(10,-1)
    \psarc(12,4){2}{0}{180}
    \psarcn(16,4){2}{0}{180}
    \psarcn(19,2){3}{0}{180}
    \psline(18,5)(18,4)
    \psline(22,5)(22,2)
    \rput{90}(20,5){\Large $\left.\right\}$} 
    \psdots(16,2)
    \psline(6,-1)(6,7)(10,11)(10,13)
    \psline(20,6)(20,7)(16,11)(16,13)
    \rput{90}(13,13){\Huge $\left.\right\}$}
    \psline(13,14)(13,17)(16,20)(16,22)
    \psline(13,17)(10,20)
    \psarcn(10,22){2}{0}{180}
    \psline(8,22)(8,28)
    \rput{90}(14,22){\Large $\left.\right\}$}
    \psline(14,23)(14,26)(12,28)
    \psline(14,26)(16,28)
    \psdots(10,20)
    \psframe[fillstyle=solid,fillcolor=boxcol](11.5,15.5)(14.5,18.5)
    \rput(13,17){\tiny $-$}
    \psframe[fillstyle=solid,fillcolor=boxcol](12.5,24.5)(15.5,27.5)
    \rput(14,26){\tiny $+$}    
  \end{pspicture}
  {~$\stackrel{B}=$~}
  \begin{pspicture}[shift=-18](6,-6)(22,28)
    \psline(8,-3)(8,-6)
    \psdots(8,-3)
    \psarcn(8,-1){2}{0}{180}
    \psline(10,4)(10,-1)
    \psarc(12,4){2}{0}{180}
    \psarcn(16,4){2}{0}{180}
    \psarcn(19,2){3}{0}{180}
    \psline(22,4)(22,2)
    \psline(18,5)(18,4)
    \psline(22,5)(22,2)
    \rput{90}(20,5){\Large $\left.\right\}$} 
    \psdots(16,2)
    \psline(6,-1)(6,7)(10,11)(10,13)
    \psline(20,6)(20,7)(16,11)(16,13)
    \rput{90}(13,13){\Huge $\left.\right\}$}
    \psline(13,14)(13,17) 
    \psdots(13,17)(16,20)
    \psarcn(13,20){3}{0}{180}
    \psline(10,20)(10,22)
    \psarcn(16,22){2}{0}{180}
    \psline(18,22)(18,28)
    \rput{90}(12,22){\Large $\left.\right\}$}
    \psline(12,23)(12,26)(10,28)
    \psline(12,26)(14,28)
    \psframe[fillstyle=solid,fillcolor=boxcol](10.5,24.5)(13.5,27.5)
    \rput(12,26){\tiny $-$}
  \end{pspicture}
  {~$=$~}
  \begin{pspicture}[shift=-8](6,-6)(22,18)
    \psline(8,-3)(8,-6)
    \psdots(8,-3)
    \psarcn(8,-1){2}{0}{180}
    \psline(10,4)(10,-1)
    \psarc(12,4){2}{0}{180}
    \psarcn(16,4){2}{0}{180}
    \psarcn(19,2){3}{0}{180}
    \psline(22,2)(22,18)
    \psdots(16,2)
    \psline(6,-1)(6,6)(10,10)(10,12)
    \psline(18,4)(18,6)(14,10)(14,12)
    \rput{90}(12,12){\Large $\left.\right\}$}
    \psline(12,13)(12,16)(10,18)
    \psline(12,16)(14,18)
    \psframe[fillstyle=solid,fillcolor=boxcol](10.5,14.5)(13.5,17.5)
    \rput(12,16){\tiny $-$}
  \end{pspicture}
\end{center}
and, by an almost identical argument, one can also derive 
  \begin{pspicture}[shift=-9](-2,2)(11,20) 
    \psline(2,15)(2,18)(4,20)
    \psline(2,18)(0,20)
    \rput{90}(2,14){\Huge $\left.\right\}$}
    \psline(5,14)(5,12)
    \psdots(5,12)
    \psarc(5,10){2}{0}{180}
    \psline(3,10)(3,8)(1,6)(-1,8)(-1,14)
    \psline(1,6)(1,2)
    \psframe[fillstyle=solid,fillcolor=boxcol](3.5,16.5)(0.5,19.5)
    \rput(2,18){\tiny $+$}    
    \psarcn(9,10){2}{0}{180}
    \psline(11,10)(11,20)
    \psframe[fillstyle=solid,fillcolor=boxcol](-0.5,4.5)(2.5,7.5)
    \rput(1,6){\tiny $-$}
  \end{pspicture}
  = 
  \begin{pspicture}[shift=-9](-2,0)(11,18)
    \psline(4,2)(4,0)
    \psdots(4,2)
    \psarc(0,8){2}{0}{180}
    \psarcn(4,8){2}{0}{180}
    \psarcn(4,8){6}{0}{180}
    \psline(0,10)(0,18)
    \psdots(0,10)
    \psline(10,8)(10,12)
    \psline(6,8)(6,12)
    \rput{90}(8,12){\Large $\left.\right\}$}
    \psline(8,13)(8,16)(10,18)
    \psline(8,16)(6,18)
    \psframe[fillstyle=solid,fillcolor=boxcol](9.5,14.5)(6.5,17.5)
    \rput(8,16){\tiny $-$}
  \end{pspicture}.
Hence, 
\begin{center}
\end{center}
\begin{center}
  \begin{pspicture}[shift=-17](-1,-1)(32,36)
    \psline(0,36)(0,24)
    \psarcn(3,24){3}{0}{180}
    \psline(6,24)(6,26)
    \psline(4,28)(6,26)(8,28)
    \psarc(10,28){2}{0}{180}
    \psarc(10,28){6}{0}{180}
    \psline(12,28)(12,17)
    \psline(16,28)(16,21)
    \psframe[fillstyle=solid,fillcolor=boxcol](4.5,24.5)(7.5,27.5)
    \rput(6,26){\tiny $-$}
    \psline(15,16)(15,4)
    \psarcn(18,4){3}{0}{180}
    \psline(21,4)(21,6)
    \psline(19,8)(21,6)(23,8)
    \psarc(25,8){2}{0}{180}
    \psarc(25,8){6}{0}{180}
    \psline(27,8)(27,-1)
    \psline(31,8)(31,-1)
    \psframe[fillstyle=solid,fillcolor=boxcol](19.5,4.5)(22.5,7.5)
    \rput(21,6){\tiny $+$}
    \psarcn(18,21){2}{0}{180}
    \psdots(18,19)
    \psline(18,19)(18,17)
    \psline(20,21)(20,36)
    \rput{270}(15,17){\Huge $\left.\right\}$}
  \end{pspicture}
  {~$=$~}
  \begin{pspicture}[shift=-6](0,1)(32,22)
    \psline(0,22)(0,4)
    \psarcn(3,4){3}{0}{180}
    \psline(6,4)(6,6)
    \psline(4,13)(4,8)(6,6)(8,8)
    \psarc(10,8){2}{0}{180}
    \psarc(9,13){5}{0}{180}
    \psarcn(16,13){2}{0}{180}
    \psline(18,13)(18,22)
    \psdots(16,11)
    \psline(16,8)(16,11)
    \psframe[fillstyle=solid,fillcolor=boxcol](4.5,4.5)(7.5,7.5)
    \rput(6,6){\tiny $-$}
    \psarcn(18,8){2}{0}{180}
    \psarcn(18,8){6}{0}{180}
    \rput{90}(22,8){\Large $\left.\right\}$}
    \psline(22,9)(22,14)
    \psline(20,16)(22,14)(24,16)
    \psarc(26,16){2}{0}{180}
    \psarc(26,16){6}{0}{180}
    \psline(28,16)(28,1)
    \psline(32,16)(32,1)
    \psframe[fillstyle=solid,fillcolor=boxcol](20.5,12.5)(23.5,15.5)
    \rput(22,14){\tiny $+$}
  \end{pspicture}
  {~$=$~}
  \begin{pspicture}[shift=-12](-4,-6)(24,26)
    \psline(8,0)(8,4)(0,12)(0,14)
    \psline(8,4)(16,12)(16,14)
    \psarc(2,14){2}{0}{180}
    \psarcn(6,14){2}{0}{180}
    \psline(8,14)(8,26)
    \psarcn(9,12){3}{0}{180}
    \psline(12,12)(12,14)
    \rput{90}(14,14){\Large $\left.\right\}$}
    \psline(14,15)(14,18)(12,20)
    \psline(14,18)(16,20)
    \psdots(6,12)
    \psframe[fillstyle=solid,fillcolor=boxcol](6.5,2.5)(9.5,5.5)
    \rput(8,4){\tiny $-$}
    \psframe[fillstyle=solid,fillcolor=boxcol](12.5,16.5)(15.5,19.5)
    \rput(14,18){\tiny $+$}    
    \psarc(18,20){2}{0}{180}
    \psarc(18,20){6}{0}{180}
    \psline(20,20)(20,-6)
    \psline(24,20)(24,-6)
    \psarcn(2,0){6}{0}{180} 
    \psline(-4,0)(-4,26)
    \psframe[linestyle=dotted](-2,0)(18,22)
  \end{pspicture}
  {~$=$~}
  \begin{pspicture}[shift=-14](2,-6)(30,28)
    \psdots(8,-3)
    \psarcn(8,-1){2}{0}{180}
    \psline(10,4)(10,-1)
    \psarc(12,4){2}{0}{180}
    \psarcn(16,4){2}{0}{180}
    \psarcn(19,2){3}{0}{180}
    \psline(22,2)(22,22)
    \psdots(16,2)
    \psline(6,-1)(6,6)(10,10)(10,12)
    \psline(18,4)(18,6)(14,10)(14,12)
    \rput{90}(12,12){\Large $\left.\right\}$}
    \psline(12,13)(12,16)(6,22)(6,28)
    \psline(12,16)(18,22)
    \psframe[fillstyle=solid,fillcolor=boxcol](10.5,14.5)(13.5,17.5)
    \rput(12,16){\tiny $-$}
    \psarc(24,22){2}{0}{180}
    \psarc(24,22){6}{0}{180}
    \psline(26,22)(26,-6)
    \psline(30,22)(30,-6)
    \psline(2,-3)(2,28)
    \psarcn(5,-3){3}{0}{180}
    \psframe[linestyle=dotted](4,-5)(24,19)
  \end{pspicture}
  {~$=$~}
  \begin{pspicture}[shift=-14](-2,-6)(22,24)
    \psarc(12,4){2}{0}{180}
    \psarcn(16,4){2}{0}{180}
    \psline(22,-2)(22,22)
    \psline(0,6)(0,24)
    \psdots(16,2)(0,6)
    \psline(16,2)(16,-2)
    \psarc(0,4){2}{0}{180}
    \psarcn(4,4){2}{0}{180}
    \psarcn(4,4){6}{0}{180}
    \psline(6,4)(6,6)(10,10)(10,12)
    \psline(18,4)(18,6)(14,10)(14,12)
    \rput{90}(12,12){\Large $\left.\right\}$}
    \psline(12,13)(12,16)(6,22)(6,24)
    \psline(12,16)(18,22)
    \psframe[fillstyle=solid,fillcolor=boxcol](10.5,14.5)(13.5,17.5)
    \rput(12,16){\tiny $-$}
    \psarc(20,22){2}{0}{180}
  \end{pspicture}
  {~$=$~}
  \begin{pspicture}[shift=-9](-3,0)(14,20)
    \psline(4,2)(4,0)
    \psdots(4,2)
    \psarc(0,8){2}{0}{180}
    \psarcn(4,8){2}{0}{180}
    \psarcn(4,8){6}{0}{180}
    \psline(0,10)(0,20)
    \psdots(0,10)
    \psline(10,8)(10,12)
    \psline(6,8)(6,12)
    \rput{90}(8,12){\Large $\left.\right\}$}
    \psline(8,13)(8,16)(10,18)
    \psline(8,16)(6,18)(6,20)
    \psframe[fillstyle=solid,fillcolor=boxcol](9.5,14.5)(6.5,17.5)
    \rput(8,16){\tiny $-$}
    \psarc(12,18){2}{0}{180}
    \psline(14,18)(14,0)
    \psframe[linestyle=dotted](-3,0.5)(12,18.5)
  \end{pspicture}
  {~$=$~}
  \begin{pspicture}[shift=-9](-3,2)(15,22) 
    \psline(2,15)(2,18)(4,20)(4,22)
    \psline(2,18)(0,20)(0,22)
    \rput{90}(2,14){\Huge $\left.\right\}$}
    \psline(5,14)(5,12)
    \psdots(5,12)
    \psarc(5,10){2}{0}{180}
    \psline(3,10)(3,8)(1,6)(-1,8)(-1,14)
    \psline(1,6)(1,2)
    \psframe[fillstyle=solid,fillcolor=boxcol](3.5,16.5)(0.5,19.5)
    \rput(2,18){\tiny $+$}    
    \psarcn(9,10){2}{0}{180}
    \psline(11,10)(11,20)
    \psframe[fillstyle=solid,fillcolor=boxcol](-0.5,4.5)(2.5,7.5)
    \rput(1,6){\tiny $-$}
    \psarc(13,20){2}{0}{180}
    \psline(15,20)(15,2)
    \psframe[linestyle=dotted](-3,2.5)(12,20.5)
  \end{pspicture}
  {~$=$~}
  \begin{pspicture}[shift=-9](-2,2)(7,22) 
    \psline(2,15)(2,18)(4,20)(4,22)
    \psline(2,18)(0,20)(0,22)
    \rput{90}(2,14){\Huge $\left.\right\}$}
    \psline(5,14)(5,12)
    \psdots(5,12)
    \psarc(5,10){2}{0}{180}
    \psline(3,10)(3,8)(1,6)(-1,8)(-1,14)
    \psline(1,6)(1,2)
    \psframe[fillstyle=solid,fillcolor=boxcol](3.5,16.5)(0.5,19.5)
    \rput(2,18){\tiny $+$}    
    \psframe[fillstyle=solid,fillcolor=boxcol](-0.5,4.5)(2.5,7.5)
    \rput(1,6){\tiny $-$}
    \psline(7,10)(7,2)
  \end{pspicture}
  \begin{pspicture}(0,0)(0,18)    
  \end{pspicture} 
\end{center}
as desired.
\end{proof}

\let\c\cedille

\end{document}